\documentclass[10pt,reqno]{amsart}

\usepackage[utf8x]{inputenc}
\usepackage[english]{babel}
\usepackage{amsmath,amsfonts,amssymb,amsthm,shuffle}
\usepackage[T1]{fontenc}
\usepackage{lmodern}

\usepackage[top=3.5cm,bottom=3.5cm,left=3.6cm,right=3.6cm]{geometry}

\usepackage[dvipsnames]{xcolor}
\usepackage[hyperindex=true,frenchlinks=true,colorlinks=true,
citecolor=Mahogany,linkcolor=DarkOrchid,urlcolor=Tan,linktocpage,
pagebackref = true]{hyperref}

\usepackage{tikz}
\usetikzlibrary{shapes}
\usetikzlibrary{fit}
\usetikzlibrary{decorations.pathmorphing}

\usepackage{dsfont}
\usepackage{wasysym}
\usepackage{cite}
\usepackage{subfigure}
\usepackage{multirow}
\usepackage{enumitem}
\usepackage{multicol}

\title[Pluriassociative algebras I]{Pluriassociative algebras I: \\
The pluriassociative operad}
\date{\today}
\author{Samuele Giraudo}
\address{Laboratoire d'Informatique Gaspard-Monge, Université Paris-Est
    Marne-la-Vallée, 5 boulevard Descartes, Champs-sur-Marne,
    77454 Marne-la-Vallée cedex 2, France}
\email{samuele.giraudo@u-pem.fr}

\numberwithin{equation}{subsection}
\renewcommand{\leq}{\leqslant}
\renewcommand{\geq}{\geqslant}

\newtheorem{Theoreme}{Theorem}[subsection]
\newtheorem{Proposition}[Theoreme]{Proposition}
\newtheorem{Lemme}[Theoreme]{Lemma}

\newcommand{\Aca}{\mathcal{A}}

\newcommand{\Fca}{\mathcal{F}}

\newcommand{\Pca}{\mathcal{P}}
\newcommand{\Hca}{\mathcal{H}}
\newcommand{\Oca}{\mathcal{O}}
\newcommand{\Mca}{\mathcal{M}}

\newcommand{\Efr}{\mathfrak{e}}
\newcommand{\Rfr}{\mathfrak{r}}
\newcommand{\Sfr}{\mathfrak{s}}
\newcommand{\Tfr}{\mathfrak{t}}

\newcommand{\Ksf}{{\mathsf{K}}}
\newcommand{\Tbb}{\mathbb{T}}
\newcommand{\Pbb}{\mathbb{P}}
\newcommand{\La}{\mathtt{a}}
\newcommand{\Lb}{\mathtt{b}}
\newcommand{\Lc}{\mathtt{c}}

\newcommand{\K}{\mathbb{K}}
\newcommand{\EnsNat}{\mathbb{N}}
\newcommand{\Gen}{\mathbb{G}}
\newcommand{\GenLibre}{\mathfrak{G}}
\newcommand{\RelLibre}{\mathfrak{R}}
\newcommand{\GenDias}{\GenLibre_{\Dias_\gamma}}

\newcommand{\GenTrias}{\GenLibre_{\Trias_\gamma}}
\newcommand{\RelDias}{\RelLibre_{\Dias_\gamma}}

\newcommand{\RelTrias}{\RelLibre_{\Trias_\gamma}}

\newcommand{\Alg}{\Aca}
\newcommand{\AlgLibre}{\Fca}
\newcommand{\Dias}{\mathsf{Dias}}
\newcommand{\Dendr}{\mathsf{Dendr}}
\newcommand{\As}{\mathsf{As}}

\newcommand{\Trias}{\mathsf{Trias}}

\newcommand{\AlgPos}{\mathsf{Pos}}
\newcommand{\AlgEns}{\mathsf{Sets}}
\newcommand{\AlgMots}{\mathsf{Words}}
\newcommand{\AlgMotsCom}{\mathsf{CWords}}
\newcommand{\AlgMotsMarq}{\mathsf{MWords}}

\newcommand{\Unite}{\mathds{1}}
\newcommand{\OpLibre}{\mathbf{Free}}
\newcommand{\Eval}{\mathrm{eval}}
\newcommand{\T}{\mathsf{T}}
\newcommand{\Vect}{\mathrm{Vect}}
\newcommand{\GDias}{\dashv}
\newcommand{\DDias}{\vdash}
\newcommand{\GDiasA}{\rotatebox[origin=c]{180}{$\Vdash$}}
\newcommand{\DDiasA}{\Vdash}

\newcommand{\MAs}{\star}

\newcommand{\MTrias}{\perp}

\newcommand{\Max}{\uparrow}
\newcommand{\Augm}{\mathrm{h}}
\newcommand{\Rel}{\leftrightarrow}
\newcommand{\Congr}{\equiv}
\newcommand{\Recr}{\to}
\newcommand{\Racine}{\mathrm{root}}
\newcommand{\Mot}{\mathrm{word}}
\newcommand{\MotT}{\mathrm{wordt}}
\newcommand{\Equerre}{\mathrm{hook}}
\newcommand{\EquerreT}{\mathrm{hookt}}
\newcommand{\OrdDias}{\preccurlyeq}
\newcommand{\Halo}{\mathrm{Halo}}

\newcommand{\MProjVersPluri}{\mathrm{M}}
\newcommand{\Incr}{\mathrm{Incr}}
\newcommand{\Hamming}{\mathrm{ham}}
\newcommand{\Corolle}{\mathrm{cor}}
\newcommand{\Action}{\cdot}

\newcommand{\Cacher}[1]{}
\newcommand{\Sloane}[1]{\href{http://oeis.org/#1}{{\bf #1}}}

\definecolor{Noir}{RGB}{0,0,0}
\definecolor{Blanc}{RGB}{255,255,255}
\definecolor{Rouge}{RGB}{205,35,38}
\definecolor{Bleu}{RGB}{2,60,195}
\definecolor{Vert}{RGB}{23,163,1}
\definecolor{Violet}{RGB}{181,18,225}
\definecolor{Orange}{RGB}{255,113,15}
\definecolor{Marron}{RGB}{52,46,0}

\tikzstyle{Noeud}=[circle,draw=Bleu!80,fill=Bleu!8,inner sep=1pt,
minimum size=2.5mm,line width=1.25pt,font=\scriptsize]
\tikzstyle{Arete}=[Rouge!80,cap=round,line width=1.25pt]
\tikzstyle{Feuille}=[rectangle,draw=Noir!70,fill=Noir!16,
inner sep=0pt,minimum size=1mm,line width=1.25pt]
\tikzstyle{Clair}=[fill=Blanc]
\tikzstyle{Marque1}=[draw=Vert!100,fill=Vert!30]
\tikzstyle{Marque2}=[draw=Orange!100,fill=Orange!40]
\tikzstyle{Marque3}=[draw=Rouge!100,fill=Rouge!50]
\tikzstyle{EtiqArete}=[regular polygon,regular polygon sides=6,
draw=Vert!80,fill=Vert!2,inner sep=.2pt,line width=.75pt,
minimum size=2.8mm,midway,text=Noir,font=\tiny]
\tikzstyle{NoeudSchr}=[Noeud,draw=Vert!80,fill=Vert!8]
\tikzstyle{NoeudCor}=[Noeud,draw=Marron!80,fill=Marron!8]

\begin{document}

\maketitle 

\begin{abstract}
    Diassociative algebras form a categoy of algebras recently
    introduced by Loday. A diassociative algebra is a vector space
    endowed with two associative binary operations satisfying some very
    natural relations. Any diassociative algebra is an algebra over the
    diassociative operad, and, among its most notable properties, this
    operad is the Koszul dual of the dendriform operad. We introduce here,
    by adopting the point of view and the tools offered by the theory of
    operads, a generalization on a nonnegative integer parameter $\gamma$
    of diassociative algebras, called $\gamma$-pluriassociative algebras,
    so that $1$-pluriassociative algebras are diassociative algebras.
    Pluriassociative algebras are vector spaces endowed with $2\gamma$
    associative binary operations satisfying some relations. We provide
    a complete study of the $\gamma$-pluriassociative operads, the
    underlying operads of the category of $\gamma$-pluriassociative
    algebras. We exhibit a realization of these operads, establish
    several presentations by generators and relations, compute their
    Hilbert series, show that they are Koszul, and construct the free
    objects in the corresponding categories. We also study several
    notions of units in $\gamma$-pluriassociative algebras and propose
    a general way to construct such algebras. This paper ends with the
    introduction of an analogous generalization of the triassociative
    operad of Loday and Ronco.
\end{abstract}

\begin{footnotesize}
    \tableofcontents
\end{footnotesize}

\section*{Introduction}
In the recent years, several algebraic structures on vector spaces
based on various sets of combinatorial objects and endowed with more or
less complicated operations on these have been considered by algebraic
combinatorists. As most famous examples, we can cite free pre-Lie
algebras~\cite{CL01}, which are vector spaces of rooted trees endowed
with a grafting product, and free dendriform algebras~\cite{Lod01}, which
are vector spaces of binary trees endowed with two products operating by
shuffling binary trees. Other well-known examples include free Zinbiel
algebras~\cite{Lod95,Lod01} endowing the space of all permutations with
a shuffle product, nonassociative permutative algebras~\cite{MY91,Liv06}
endowing the space of all rooted trees with a grafting product at the
root, and duplicial algebras~\cite{Lod08} endowing the space of all
binary with two grafting operations.
\medskip

Instead of studying all these algebraic structures separately, it is
possible to ask and treat some general questions about these under a
uniform point of view. The theory of operads is an efficient tool
to regard different categories of algebraic structures in a unified
manner. This theory (see~\cite{LV12} for a complete exposition and
also~\cite{Cha08} for an exposition highlighting the combinatorial
aspects of the theory) has been introduced in the context of algebraic
topology~\cite{May72,BV73}. Roughly speaking, an operad is a space
of abstract operators consisting in several inputs and one output
that can be composed to form bigger ones. The point is that any operad
encodes a category of algebras and working with an operad amounts to work
with the algebras all together of this category. Moreover, the use of the
theory of operads leads to the discovery of connections between differents
sorts of algebras by terms of morphisms of operads. As a simple example,
the well-known fact that any associative algebra gives rise to a Lie
algebra by considering its associator as a Lie bracket comes from the
fact that there is a morphism from the underlying operad of the category
of Lie algebras to the underlying operad of the category of associative
algebras.
\medskip

The present work is concerned with the definition of a coherent
generalization of dialgebras, algebraic structures introduced by Loday
in~\cite{Lod01}. A dialgebra is a vector space endowed with two
associative binary operations $\GDias$ and $\DDias$ satisfying some
relations. From a combinatorial point of view, the bases of the free
dialgebra over one generator are indexed by ordered pairs $(n, k)$ of
integers, denoted by $\Efr_{n, k}$, and satisfying $1 \leq k \leq n$.
The operations $\GDias$ and $\DDias$ admit simple set-theoretic
descriptions over this basis~\cite{Cha05}. In a previous
work~\cite{Gir12,Gir15}, we introduced a new construction for the operad
$\Dias$, the underlying operad of the category of diassociative algebras,
and we raised the question whether this construction can be extended to
obtain operads generalizing $\Dias$ and hence, to obtain generalizations
of dialgebras.
\medskip

Let us give some explanations about our construction of $\Dias$.
In~\cite{Gir12,Gir15}, we defined a general functorial construction $\T$
producing an operad from any monoid. This construction $\T$ sends a
monoid $M$ to the operad $\T M$ of all words on $M$, where $M$ is seen
as an alphabet. The arity of a word is its length and the operadic partial
composition $u \circ_i v$ of two words $u$ and $v$ of $\T M$ consists in
replacing the $i$th letter $u_i$ of $u$ by a version of $v$ obtained by
multpliying to the left all its letters by $u_i$. The operad $\Dias$ is
the suboperad of $\T M$, where $M$ is the multiplicative monoid on
$\{0, 1\}$, generated by the two words $01$ and $10$ of arity two. In
the present paper, we rely on $\T$ to construct a generalization on a
nonnegative integer parameter $\gamma$ of $\Dias$, denoted by
$\Dias_\gamma$, in such a way that $\Dias_1 = \Dias$ and $\Dias_\gamma$ is
a suboperad of $\Dias_{\gamma + 1}$ for any $\gamma \geq 0$. The operads
$\Dias_\gamma$, called $\gamma$-pluriassociative operads, are set-operads
involving words on the alphabet $\{0, 1, \dots, \gamma\}$ with exactly
one occurrence of~$0$. Besides, this work naturally leads to the
consideration and the definition of several new operads.
Table~\ref{tab:operades_introduites_I} summarizes some information about
these. We provide for instance a generalization on a nonnegative integer
parameter $\gamma$ of the triassociative operad $\Trias$~\cite{LR04},
denoted by $\Trias_\gamma$.
\begin{table}
    \centering
    \begin{tabular}{c|c|c}
        Operad & Objects & Dimensions \\ \hline \hline
        $\Dias_\gamma$ & Words on $\{0, 1, \dots, \gamma\}$
        with exactly one $0$
            & $n \gamma^{n - 1}$ \\ \hline
        $\As_\gamma$ & $\gamma$-corollas & $\gamma$ \\ \hline
        $\Trias_\gamma$ & Words on $\{0, 1, \dots, \gamma\}$
        with at least one $0$
            & $(\gamma + 1)^n - \gamma^n$
    \end{tabular}
    \bigskip
    \caption{The main operads defined in this paper. All these operads
    depend on a nonnegative integer parameter $\gamma$. The shown
    dimensions are the ones of the homogeneous components of
    arities $n \geq 2$ of the operads.}
    \label{tab:operades_introduites_I}
\end{table}
\medskip

The main rationale for this work is to establish the necessary
foundations to propose a generalization on a nonnegative integer
parameter $\gamma$ of dendriform algebras~\cite{Lod01}. Since $\Dias$ is
the Koszul dual~\cite{GK94} of the operad $\Dendr$, the underlying
operad of the category of dendriform algebras, our objective is to
propose the definition of the operads $\Dendr_\gamma$, defined each as
the Koszul dual of $\Dias_\gamma$. Moreover, since $\Dias$ admits a
description far simpler than $\Dendr$, starting by constructing a
generalization of $\Dias$ to obtain a generalization of $\Dendr$ by
Koszul duality is a convenient path to explore. This strategy is
developed in the continutation of this work~\cite{GirII}, where the
operads $\Dendr_\gamma$ are studied. This lead to new sorts of algebras,
providing analogs of dendriform algebras and different from already
existing ones (see for instance~\cite{LR04,AL04,Ler04,Ler07,Nov14}).
\medskip

This paper is organized as follows. Section~\ref{sec:outils} contains a
conspectus of the tools used in this paper. We recall here the
definition of the construction $\T$~\cite{Gir12,Gir15} and provide a
reformulation of results of Hoffbeck~\cite{Hof10} and Dotsenko and
Khoroshkin~\cite{DK10} to prove that an operad is Koszul by using
convergent rewrite rules. Besides, this part provides self-contained
definitions about nonsymmetric operads, algebras over operads, free
operads, and rewrite rules on trees. This section ends by some recalls
about the diassociative operad and diassociative algebras.
\medskip

Section~\ref{sec:dias_gamma} is devoted to the introduction and the
study of the operad $\Dias_\gamma$. We begin by detailing the
construction of $\Dias_\gamma$ as a suboperad of the operad obtained by
the construction $\T$ applied on the monoid $\Mca_\gamma$ with $\{0, 1, \dots, \gamma\}$
as underlying set and with the operation $\max$ as product. More
precisely, $\Dias_\gamma$ is defined as the suboperad of $\T \Mca_\gamma$
generated by the words $0a$ and $a0$ for all $a \in \{1, \dots, \gamma\}$.
We then provide a presentation by generators and relations of $\Dias_\gamma$
(Theorem~\ref{thm:presentation_dias_gamma}), and show that it is a
Koszul operad (Theorem~\ref{thm:koszulite_dias_gamma}). We also establish
some more properties of this operad: we compute its group of symmetries
(Proposition~\ref{prop:symetries_dias_gamma}), show that it is a basic
operad in the sense of~\cite{Val07}
(Proposition~\ref{prop:dias_gamma_basique}), and show that it is a
rooted operad in the sense of~\cite{Cha14}
(Proposition~\ref{prop:dias_gamma_basique}). We end this section by
introducing an alternating basis of $\Dias_\gamma$, the $\Ksf$-basis,
defined through a partial ordering relation over the words indexing the
bases of $\Dias_\gamma$. After describing how the partial composition of
$\Dias_\gamma$ expresses over the $\Ksf$-basis
(Theorem~\ref{thm:composition_base_K}), we provide a presentation of
$\Dias_\gamma$ over this basis
(Proposition~\ref{prop:presentation_alternative_dias_gamma}). Despite
the fact that this alternative presentation is more complex than the
original one of $\Dias_\gamma$ provided by
Theorem~\ref{thm:presentation_dias_gamma}, the computation of the Koszul
dual $\Dendr_\gamma$ of $\Dias_\gamma$ from this second presentation
leads to a surprisingly plain presentation of $\Dendr_\gamma$
considered later in~\cite{GirII}.
\medskip

In Section~\ref{sec:algebres_dias_gamma}, algebras over $\Dias_\gamma$,
called $\gamma$-pluriassociative algebras, are studied. The free
$\gamma$-pluriassociative algebra over one generator is described as a
vector space of words on the alphabet $\{0, 1, \dots, \gamma\}$ with
exactly one occurrence of $0$, endowed with $2\gamma$ binary operations
(Proposition~\ref{prop:algebre_dias_gamma_libre}). We next study two
different notions of units in $\gamma$-pluriassociative algebras, the
bar-units and the wire-units, that are generalizations of definitions of
Loday introduced into the context of diassociative algebras~\cite{Lod01}.
We show that the presence of a wire-unit in a $\gamma$-pluriassociative
algebra leads to many consequences on its structure
(Proposition~\ref{prop:unite_fil}). Besides, we describe a general
construction $\MProjVersPluri$ to obtain $\gamma$-pluriassociative
algebras by starting from $\gamma$-multiprojection algebras, that are
algebraic structures with $\gamma$ associative products and endowed with
$\gamma$ endomorphisms with extra relations
(Theorem~\ref{thm:algebre_multiprojections_vers_dias_gamma}). The main
interest of the construction $\MProjVersPluri$ is that
$\gamma$-multiprojection algebras are simpler algebraic structures than
$\gamma$-pluriassociative algebras. The bar-units and wire-units of the
$\gamma$-pluriassociative algebras obtained by this construction are then
studied
(Proposition~\ref{prop:algebre_multiprojections_vers_dias_gamma_proprietes}).
We end this section by listing five examples of $\gamma$-pluriassociative
algebras constructed from $\gamma$-multiprojection algebras, including
the free $\gamma$-pluriassociative algebra over one generator considered
in Section~\ref{subsubsec:algebre_dias_gamma_libre}.
\medskip

Finally, by using almost the same tools as the one used in
Section~\ref{sec:dias_gamma}, we propose in Section~\ref{sec:pluritriass}
a generalization on a nonnegative integer parameter $\gamma$ of the
triassociative operad $\Trias$ of Loday and Ronco~\cite{LR04}, denoted
by $\Trias_\gamma$. This follows a very simple idea: like $\Dias_\gamma$,
$\Trias_\gamma$ is defined as a suboperad of $\T \Mca_\gamma$ generated
by the same generators as those of $\Dias_\gamma$, plus the word $00$.
In a previous work~\cite{Gir12,Gir15}, we showed that $\Trias_1$ is the
triassociative operad. We provide here an expression for the Hilbert
series of $\Trias_\gamma$ obtained from the description of its elements
(Proposition~\ref{prop:elements_trias_gamma}) and a presentation
(Theorem~\ref{thm:presentation_trias_gamma}).
\bigskip

{\it Acknowledgements.} The author would like to thank Jean-Yves-Thibon
for its pertinent remarks and questions about this work when it was in
progress. Thanks are addressed to Frederic Chapoton and Eric Hoffbeck
for answering some questions of the author respectively about the
dendriform and diassociative operads, and Koszulity of operads. The
author thanks also Vladimir Dotsenko and Bruno Vallette for pertinent
bibliographic suggestions. Finally, the author warmly thanks the referee
for his very careful reading and his suggestions, improving the quality
of the paper.
\bigskip

{\it Notations and general conventions.}
All the algebraic structures of this article have a field of characteristic
zero $\K$ as ground field. If $S$ is a set, $\Vect(S)$ denotes the linear
span of the elements of $S$. For any integers $a$ and $c$, $[a, c]$ denotes
the set $\{b \in \EnsNat : a \leq b \leq c\}$ and $[n]$, the set $[1, n]$.
The cardinality of a finite set $S$ is denoted by $\# S$. If $u$ is a
word, its letters are indexed from left to right from $1$ to its length
$|u|$. For any $i \in [|u|]$, $u_i$ is the letter of $u$ at position $i$.
If $a$ is a letter and $n$ is a nonnegative integer, $a^n$ denotes the
word consisting in $n$ occurrences of $a$. Notice that $a^0$ is the
empty word $\epsilon$.
\medskip

\section{Preliminaries: algebraic structures and main tools}%
\label{sec:outils}
This preliminary section sets our conventions and notations about operads
and algebras over an operad, and describes the main tools we will use.
The definitions and some properties of the diassociative operad are also
recalled. This section does not contains new results but it is a
self-contained set of definitions about operads intended to readers
familiar with algebra or combinatorics but not necessarily with operadic
theory.
\medskip

\subsection{Operads and algebras over an operad}
We list here several staple definitions about operads and algebras over
an operad. We present also an important tool for this work: the
construction $\T$ producing operads from monoids.
\medskip

\subsubsection{Operads}
A {\em nonsymmetric operad in the category of vector spaces}, or a
{\em nonsymmetric operad} for short, is a graded vector space
$\Oca := \bigoplus_{n \geq 1} \Oca(n)$ together with linear maps
\begin{equation}
    \circ_i : \Oca(n) \otimes \Oca(m) \to \Oca(n + m - 1),
    \qquad n, m \geq 1, i \in [n],
\end{equation}
called {\em partial compositions}, and a distinguished element
$\Unite \in \Oca(1)$, the {\em unit} of $\Oca$. This data has to satisfy
the three relations
\begin{subequations}
\begin{equation}
    (x \circ_i y) \circ_{i + j - 1} z = x \circ_i (y \circ_j z),
    \qquad x \in \Oca(n), y \in \Oca(m),
    z \in \Oca(k), i \in [n], j \in [m],
\end{equation}
\begin{equation}
    (x \circ_i y) \circ_{j + m - 1} z = (x \circ_j z) \circ_i y,
    \qquad x \in \Oca(n), y \in \Oca(m),
    z \in \Oca(k), i < j \in [n],
\end{equation}
\begin{equation}
    \Unite \circ_1 x = x = x \circ_i \Unite,
    \qquad x \in \Oca(n), i \in [n].
\end{equation}
\end{subequations}
Since we shall consider in this paper mainly nonsymmetric operads, we
shall call these simply {\em operads}. Moreover, all considered operads
are such that $\Oca(1)$ has dimension~$1$.
\medskip

If $x$ is an element of $\Oca$ such that $x \in \Oca(n)$ for a $n \geq 1$,
we say that $n$ is the {\em arity} of $x$ and we denote it by $|x|$. An
element $x$ of $\Oca$ of arity $2$ is {\em associative} if
$x \circ_1 x = x \circ_2 x$. If $\Oca_1$ and $\Oca_2$ are operads, a
linear map $\phi : \Oca_1 \to \Oca_2$ is an {\em operad morphism} if it
respects arities, sends the unit of $\Oca_1$ to the unit of $\Oca_2$,
and commutes with partial composition maps. We say that $\Oca_2$ is a
{\em suboperad} of $\Oca_1$ if $\Oca_2$ is a graded subspace of $\Oca_1$,
and $\Oca_1$ and $\Oca_2$ have the same unit and the same partial
compositions. For any set $G \subseteq \Oca$, the {\em operad generated
by} $G$ is the smallest suboperad of $\Oca$ containing $G$. When the
operad generated by $G$ is $\Oca$ itself and $G$ is minimal with respect
to inclusion among the subsets of $\Oca$ satisfying this property, $G$
is a {\em generating set} of $\Oca$ and its elements are {\em generators}
of $\Oca$. An {\em operad ideal} of $\Oca$ is a graded subspace $I$ of
$\Oca$ such that, for any $x \in \Oca$ and $y \in I$, $x \circ_i y$ and
$y \circ_j x$ are in $I$ for all valid integers $i$ and $j$. Given an
operad ideal $I$ of $\Oca$, one can define the {\em quotient operad}
$\Oca/_I$ of $\Oca$ by $I$ in the usual way. When $\Oca$ is such that
all $\Oca(n)$ are finite for all $n \geq 1$, the {\em Hilbert series}
of $\Oca$ is the series $\Hca_\Oca(t)$ defined by
\begin{equation}
    \Hca_\Oca(t) := \sum_{n \geq 1} \dim \Oca(n) t^n.
\end{equation}
\medskip

Instead of working with the partial composition maps of $\Oca$, it is
something useful to work with the maps
\begin{equation}
    \circ : \Oca(n) \otimes \Oca(m_1) \otimes \dots \otimes \Oca(m_n)
    \to \Oca(m_1 + \dots + m_n),
    \qquad n, m_1, \dots, m_n \geq 1,
\end{equation}
linearly defined for any $x \in \Oca$ of arity $n$ and
$y_1, \dots, y_{n - 1}, y_n \in \Oca$ by
\begin{equation}
    x \circ (y_1, \dots, y_{n - 1}, y_n) :=
    (\dots ((x \circ_n y_n) \circ_{n - 1} y_{n - 1}) \dots) \circ_1 y_1.
\end{equation}
These maps are called {\em composition maps} of $\Oca$.
\medskip

\subsubsection{Set-operads}
Instead of being a direct sum of vector spaces $\Oca(n)$, $n \geq 1$,
$\Oca$ can be a graded disjoint union of sets. In this context, $\Oca$
is a {\em set-operad}. All previous definitions remain valid by replacing
direct sums $\oplus$ by disjoint unions $\sqcup$, tensor products
$\otimes$ by Cartesian products $\times$, and vector space dimensions
$\dim$ by set cardinalities $\#$. Moreover, in the context of set-operads,
we work with {\em operad congruences} instead of operad ideals. An operad
congruence on a set-operad $\Oca$ is an equivalence relation $\Congr$ on
$\Oca$ such that all elements of a same $\Congr$-equivalence class have
the same arity and for all elements $x$, $x'$, $y$, and $y'$ of $\Oca$,
$x \Congr x'$ and $y \Congr y'$ imply $x \circ_i y \Congr x' \circ_i y'$
for all valid integers $i$. The {\em quotient operad} $\Oca/_\Congr$ of
$\Oca$ by $\Congr$ is the set-operad defined in the usual way.
\medskip

Any set-operad $\Oca$ gives naturally rise to an operad on $\Vect(\Oca)$
by extending the partial compositions of $\Oca$ by linearity. Besides
this, any equivalence relation $\Rel$ of $\Oca$ such that all elements
of a same $\Rel$-equivalence class have the same arity induces a subspace
of $\Vect(\Oca)$ generated by all $x - x'$ such that $x \Rel x'$, called
{\em space induced} by $\Rel$. In particular, any operad congruence
$\Congr$ on $\Oca$ induces an operad ideal of $\Vect(\Oca)$.
\medskip

\subsubsection{From monoids to operads}%
\label{subsec:monoides_vers_operades}
In a previous work \cite{Gir12,Gir15}, the author introduced a
construction which, from any monoid, produces an operad. This
construction is described as follows. Let $\Mca$ be a monoid with an
associative product $\bullet$ admitting a unit $1$. We denote by
$\T \Mca$ the operad $\T \Mca := \bigoplus_{n \geq 1} \T \Mca(n)$ where
for all $n \geq 1$,
\begin{equation}
    \T \Mca(n) := \Vect\left(\left\{u_1 \dots u_n : u_i \in \Mca
    \mbox{ for all } i \in [n] \right\}\right).
\end{equation}
The partial composition of two words $u \in \T \Mca(n)$ and
$v \in \T \Mca(m)$ is linearly defined by
\begin{equation}
    u \circ_i v := u_1 \dots u_{i - 1}
    \, (u_i \bullet v_1) \, \dots \, (u_i \bullet v_m) \,
    u_{i + 1} \dots u_n,
    \qquad i \in [n].
\end{equation}
The unit of $\T \Mca$ is $\Unite := 1$. In other words, $\T \Mca$ is
the vector space of words on $\Mca$ seen as an alphabet and the partial
composition returns to insert a word $v$ onto the $i$th letter $u_i$ of
a word $u$ together with a left multiplication by $u_i$.
\medskip

\subsubsection{Algebras over an operad}
Any operad $\Oca$ encodes a category of algebras whose objects are
called {\em $\Oca$-algebras}. An $\Oca$-algebra $\Alg_\Oca$ is a vector
space endowed with a right action
\begin{equation}
    \Action : \Alg_\Oca^{\otimes n} \otimes \Oca(n) \to \Alg_\Oca,
    \qquad n \geq 1,
\end{equation}
satisfying the relations imposed by the structure of $\Oca$, that are
\begin{multline} \label{equ:relation_algebre_sur_operade}
    (e_1 \otimes \dots \otimes e_{n + m - 1}) \Action (x \circ_i y)
    = \\
    (e_1 \otimes \dots
        \otimes e_{i - 1} \otimes
        (e_i \otimes \dots \otimes e_{i + m - 1}) \Action y
        \otimes e_{i + m} \otimes
        \dots \otimes e_{n + m - 1}) \Action x,
\end{multline}
for all
$e_1 \otimes \dots \otimes e_{n + m - 1} \in \Alg_\Oca^{\otimes {n + m - 1}}$,
$x \in \Oca(n)$, $y \in \Oca(m)$, and $i \in [n]$. Notice that,
by~\eqref{equ:relation_algebre_sur_operade}, if $G$ is a generating set
of $\Oca$, it is enough to define the action of each $x \in G$ on
$\Alg_\Oca^{\otimes |x|}$ to wholly define~$\Action$.
\medskip

In other words, any element $x$ of $\Oca$ of arity $n$ plays the role of
a linear operation
\begin{equation}
    x : \Alg_\Oca^{\otimes n}  \to \Alg_\Oca,
\end{equation}
taking $n$ elements of $\Alg_\Oca$ as inputs and computing an element of
$\Alg_\Oca$. By a slight but convenient abuse of notation, for any
$x \in \Oca(n)$, we shall denote by $x(e_1, \dots, e_n)$, or by
$e_1 \, x \, e_2$ if $x$ has arity $2$, the element
$(e_1 \otimes \dots \otimes e_n) \Action x$ of $\Alg_\Oca$, for any
$e_1 \otimes \dots \otimes e_n \in \Alg_\Oca^{\otimes n}$.
Observe that by~\eqref{equ:relation_algebre_sur_operade}, any associative
element of $\Oca$ gives rise to an associative operation on $\Alg_\Oca$.
\medskip

Arrows in the category of $\Oca$-algebras are
{\em $\Oca$-algebra morphisms}, that are linear maps
$\phi : \Alg_1 \to \Alg_2$ between two $\Oca$-algebras $\Alg_1$ and
$\Alg_2$ such that
\begin{equation}
    \phi(x(e_1, \dots, e_n)) = x(\phi(e_1), \dots, \phi(e_n)),
\end{equation}
for all $e_1,  \dots, e_n \in \Alg_1$ and $x \in \Oca(n)$. We say that
$\Alg_2$ is an {\em $\Oca$-subalgebra} of $\Alg_1$ if $\Alg_2$ is a
subspace of $\Alg_1$ and $\Alg_1$ and $\Alg_2$ are endowed with the same
right action of $\Oca$. If $G$ is a set of elements of an $\Oca$-algebra
$\Alg$, the {\em $\Oca$-algebra generated by} $G$ is the smallest
$\Oca$-subalgebra of $\Alg$ containing $G$. When the $\Oca$-algebra
generated by $G$ is $\Alg$ itself and $G$ is minimal with respect to
inclusion among the subsets of $\Alg$ satisfying this property, $G$ is
a {\em generating set} of $\Alg$ and its elements are {\em generators}
of $\Alg$. An {\em $\Oca$-algebra ideal} of $\Alg$ is a subspace $I$ of
$\Alg$ such that for all operation $x$ of $\Oca$ of arity $n$ and elements
$e_1$, \dots, $e_n$ of $\Oca$, $x(e_1, \dots, e_n)$ is in $I$ whenever
there is a $i \in [n]$ such that $e_i$ is in $I$.
\medskip

The {\em free $\Oca$-algebra over one generator} is the $\Oca$-algebra
$\AlgLibre_\Oca$ defined in the following way. We set
$\AlgLibre_\Oca := \oplus_{n \geq 1} \AlgLibre_\Oca(n)
:= \oplus_{n \geq 1} \Oca(n)$,
and for any
$e_1, \dots, e_n \in \AlgLibre_\Oca$ and $x \in \Oca(n)$, the right
action of $x$ on $e_1 \otimes \dots \otimes e_n$ is defined by
\begin{equation}
    x(e_1, \dots, e_n) := x \circ (e_1, \dots, e_n).
\end{equation}
Then, any element $x$ of $\Oca(n)$ endows $\AlgLibre_\Oca$ with an
operation
\begin{equation}
    x : \AlgLibre_\Oca(m_1) \otimes \dots \otimes \AlgLibre_\Oca(m_n)
    \to \AlgLibre_\Oca(m_1 + \dots + m_n)
\end{equation}
respecting the graduation of $\AlgLibre_\Oca$.
\medskip

\subsection{Free operads, rewrite rules, and Koszulity}
We recall here a description of free operads through syntax trees and
presentations of operads by generators and relations. The Koszul
property for operads is a very important notion in this paper and its
sequel~\cite{GirII}. We recall it and describe an already known
criterion to prove that a set-operad is Koszul by passing by rewrite
rules on syntax trees.
\medskip

\subsubsection{Syntax trees}
Unless otherwise specified, we use in the sequel the standard
terminology ({\em i.e.}, {\em node}, {\em edge}, {\em root},
{\em parent}, {\em child}, {\em path}, {\em ancestor}, {\em etc.}) about
planar rooted trees~\cite{Knu97}. Let $\Tfr$ be a planar rooted tree.
The {\em arity} of a node of $\Tfr$ is its number of children. An
{\em internal node} (resp. a {\em leaf}) of $\Tfr$ is a node with a
nonzero (resp. null) arity. Given an internal node $x$ of $\Tfr$, due to
the planarity of $\Tfr$, the children of $x$ are totally ordered from
left to right and are thus indexed from $1$ to the arity of $x$. If $y$
is a child of $x$, $y$ defines a {\em subtree} of $\Tfr$, that is the
planar rooted tree with root $y$ and consisting in the nodes of $\Tfr$
that have $y$ as ancestor. We shall call {\em $i$th subtree} of $x$ the
subtree of $\Tfr$ rooted at the $i$th child of $x$. A {\em partial subtree}
of $\Tfr$ is a subtree of $\Tfr$ in which some internal nodes have been
replaced by leaves and its descendants has been forgotten. Besides, due
to the planarity of $\Tfr$, its leaves are totally ordered from left to
right and thus are indexed from $1$ to the arity of $\Tfr$. In our
graphical representations, each tree is depicted so that its root is the
uppermost node.
\medskip

Let $S := \sqcup_{n \geq 1} S(n)$ be a graded set. By extension, we say
that the {\em arity} of an element $x$ of $S$ is $n$ provided that
$x \in S(n)$. A {\em syntax tree on $S$} is a planar rooted tree such
that its internal nodes of arity $n$ are labeled on elements of arity
$n$ of $S$. The {\em degree} (resp. {\em arity}) of a syntax tree is its
number of internal nodes (resp. leaves). For instance, if
$S := S(2) \sqcup S(3)$ with $S(2) := \{\La, \Lc\}$ and $S(3) := \{\Lb\}$,
\begin{equation}
    \begin{split}
    \begin{tikzpicture}[xscale=.35,yscale=.17]
        \node(0)at(0.00,-6.50){};
        \node(10)at(8.00,-9.75){};
        \node(12)at(10.00,-9.75){};
        \node(2)at(2.00,-6.50){};
        \node(4)at(3.00,-3.25){};
        \node(5)at(4.00,-9.75){};
        \node(7)at(5.00,-9.75){};
        \node(8)at(6.00,-9.75){};
        \node(1)at(1.00,-3.25){\begin{math}\Lc\end{math}};
        \node(11)at(9.00,-6.50){\begin{math}\La\end{math}};
        \node(3)at(3.00,0.00){\begin{math}\Lb\end{math}};
        \node(6)at(5.00,-6.50){\begin{math}\Lb\end{math}};
        \node(9)at(7.00,-3.25){\begin{math}\La\end{math}};
        \node(r)at(3.00,2.75){};
        \draw(0)--(1); \draw(1)--(3); \draw(10)--(11); \draw(11)--(9);
        \draw(12)--(11); \draw(2)--(1); \draw(4)--(3); \draw(5)--(6);
        \draw(6)--(9); \draw(7)--(6); \draw(8)--(6); \draw(9)--(3);
        \draw(r)--(3);
    \end{tikzpicture}
    \end{split}
\end{equation}
is a syntax tree on $S$ of degree $5$ and arity $8$. Its root is labeled
by $\Lb$ and has arity $3$.
\medskip

\subsubsection{Free operads}
Let $S$ be a graded set. The {\em free operad $\OpLibre(S)$ over $S$} is
the operad wherein for any $n \geq 1$, $\OpLibre(S)(n)$ is the vector
space of syntax trees on $S$ of arity $n$, the partial composition
$\Sfr \circ_i \Tfr$ of two syntax trees $\Sfr$ and $\Tfr$ on $S$
consists in grafting the root of $\Tfr$ on the $i$th leaf of $\Sfr$, and
its unit is the tree consisting in one leaf. For instance, if
$S := S(2) \sqcup S(3)$ with $S(2) := \{\La, \Lc\}$ and
$S(3) := \{\Lb\}$, one has in $\OpLibre(S)$,
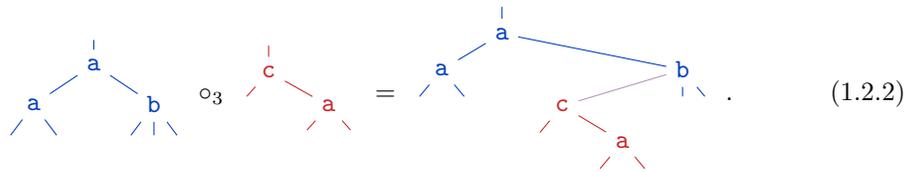
\begin{equation}\begin{split}\end{split}
    \begin{split}
    \begin{tikzpicture}[xscale=.4,yscale=.2]
        \node(0)at(0.00,-5.33){};
        \node(2)at(2.00,-5.33){};
        \node(4)at(4.00,-5.33){};
        \node(6)at(5.00,-5.33){};
        \node(7)at(6.00,-5.33){};
        \node[text=Bleu](1)at(1.00,-2.67){\begin{math}\La\end{math}};
        \node[text=Bleu](3)at(3.00,0.00){\begin{math}\La\end{math}};
        \node[text=Bleu](5)at(5.00,-2.67){\begin{math}\Lb\end{math}};
        \node(r)at(3.00,2.25){};
        \draw[draw=Bleu](0)--(1); \draw[draw=Bleu](1)--(3);
        \draw[draw=Bleu](2)--(1); \draw[draw=Bleu](4)--(5);
        \draw[draw=Bleu](5)--(3); \draw[draw=Bleu](6)--(5);
        \draw[draw=Bleu](7)--(5); \draw[draw=Bleu](r)--(3);
    \end{tikzpicture}
    \end{split}
    \circ_3
    \begin{split}
    \begin{tikzpicture}[xscale=.4,yscale=.27]
        \node(0)at(0.00,-1.67){};
        \node(2)at(2.00,-3.33){};
        \node(4)at(4.00,-3.33){};
        \node[text=Rouge](1)at(1.00,0.00){\begin{math}\Lc\end{math}};
        \node[text=Rouge](3)at(3.00,-1.67){\begin{math}\La\end{math}};
        \node(r)at(1.00,1.75){};
        \draw[draw=Rouge](0)--(1); \draw[draw=Rouge](2)--(3);
        \draw[draw=Rouge](3)--(1); \draw[draw=Rouge](4)--(3);
        \draw[draw=Rouge](r)--(1);
    \end{tikzpicture}
    \end{split}
    =
    \begin{split}
    \begin{tikzpicture}[xscale=.4,yscale=.2]
        \node(0)at(0.00,-4.80){};
        \node(10)at(9.00,-4.80){};
        \node(11)at(10.00,-4.80){};
        \node(2)at(2.00,-4.80){};
        \node(4)at(4.00,-7.20){};
        \node(6)at(6.00,-9.60){};
        \node(8)at(8.00,-9.60){};
        \node[text=Bleu](1)at(1.00,-2.40){\begin{math}\La\end{math}};
        \node[text=Bleu](3)at(3.00,0.00){\begin{math}\La\end{math}};
        \node[text=Rouge](5)at(5.00,-4.80){\begin{math}\Lc\end{math}};
        \node[text=Rouge](7)at(7.00,-7.20){\begin{math}\La\end{math}};
        \node[text=Bleu](9)at(9.00,-2.40){\begin{math}\Lb\end{math}};
        \node(r)at(3.00,2.40){};
        \draw[draw=Bleu](0)--(1); \draw[draw=Bleu](1)--(3);
        \draw[draw=Bleu](10)--(9); \draw[draw=Bleu](11)--(9);
        \draw[draw=Bleu](2)--(1); \draw[draw=Rouge](4)--(5);
        \draw[draw=Bleu!50!Rouge!50](5)--(9); \draw[draw=Rouge](6)--(7);
        \draw[draw=Rouge](7)--(5); \draw[draw=Rouge](8)--(7);
        \draw[draw=Bleu](9)--(3); \draw[draw=Bleu](r)--(3);
    \end{tikzpicture}
    \end{split}\,.
\end{equation}
\medskip

We denote by $\Corolle : S \to \OpLibre(S)$ the inclusion map, sending
any $x$ of $S$ to the {\em corolla} labeled by $x$, that is the syntax
tree consisting in one internal node labeled by $x$ attached to a
required number of leaves. In the sequel, if required by the context, we
shall implicitly see any element $x$ of $S$ as the corolla $\Corolle(x)$
of $\OpLibre(S)$. For instance, when $x$ and $y$ are two elements of $S$,
we shall simply denote by $x \circ_i y$ the syntax tree
$\Corolle(x) \circ_i \Corolle(y)$ for all valid integers $i$.
\medskip

For any operad $\Oca$, by seeing $\Oca$ as a graded set, $\OpLibre(\Oca)$
is the free operad of the syntax trees linearly labeled by elements of
$\Oca$. The {\em evaluation map} of $\Oca$ is the map
\begin{equation}
    \Eval_\Oca : \OpLibre(\Oca) \to \Oca,
\end{equation}
recursively defined by
\begin{equation}
    \Eval_\Oca(\Tfr) :=
    \begin{cases}
        \Unite & \mbox{if } \Tfr \mbox{ is the leaf}, \\
        x \circ (\Eval_\Oca(\Sfr_1), \dots, \Eval_\Oca(\Sfr_n)) &
        \mbox{otherwise},
    \end{cases}
\end{equation}
where $\Unite$ is the unit of $\Oca$, $x$ is the label of the root of
$\Tfr$, and $\Sfr_1$, \dots, $\Sfr_n$ are, from left to right, the
subtrees of the root of $\Tfr$. In other words, any tree $\Tfr$ of
$\OpLibre(\Oca)$ can be seen as a tree-like expression for an element
$\Eval_\Oca(\Tfr)$ of $\Oca$. Moreover, by induction on the degree of
$\Tfr$, it appears that $\Eval_\Oca$ is a well-defined surjective operad
morphism.
\medskip

\subsubsection{Presentations by generators and relations}
A {\em presentation} of an operad $\Oca$ consists in a pair
$(\GenLibre, \RelLibre)$ such that
$\GenLibre := \sqcup_{n \geq 1} \GenLibre(n)$ is a graded set,
$\RelLibre$ is a subspace of $\OpLibre(\GenLibre)$, and  $\Oca$ is
isomorphic to $\OpLibre(\GenLibre)/_{\langle \RelLibre \rangle}$,  where
$\langle \RelLibre \rangle$ is the operad ideal of $\OpLibre(\GenLibre)$
generated by $\RelLibre$. We call $\GenLibre$ the {\em set of generators}
and $\RelLibre$ the {\em space of relations} of $\Oca$. We say that
$\Oca$ is {\em quadratic} if one can exhibit a presentation
$(\GenLibre, \RelLibre)$ of $\Oca$ such that $\RelLibre$ is a homogeneous
subspace of $\OpLibre(\GenLibre)$ consisting in syntax trees of degree
$2$. Besides, we say that $\Oca$ is {\em binary} if one can exhibit a
presentation $(\GenLibre, \RelLibre)$ of $\Oca$ such that $\GenLibre$ is
concentrated in arity~$2$.
\medskip

With knowledge of a presentation $(\GenLibre, \RelLibre)$ of $\Oca$, it
is easy to describe the category of the $\Oca$-algebras. Indeed, by
denoting by
$\pi : \OpLibre(\GenLibre) \to \OpLibre(\GenLibre)/_{\langle \RelLibre \rangle}$
the canonical surjection map, the category of $\Oca$-algebras is the
category of vector spaces $\Alg_\Oca$ endowed with maps $\pi(g)$,
$g \in \GenLibre$, satisfying for all $r \in \RelLibre$ the relations
\begin{equation}
    r(e_1,\dots, e_n) = 0,
\end{equation}
for all $e_1, \dots, e_n \in \Alg_\Oca$, where $n$ is the arity of $r$.
\medskip

\subsubsection{Rewrite rules}
Let $S$ be a graded set. A {\em rewrite rule} on syntax trees on $S$ is
a binary relation $\Recr$ on $\OpLibre(S)$ whenever for all trees $\Sfr$
and $\Tfr$ of $\OpLibre(S)$, $\Sfr \Recr \Tfr$ only if $\Sfr$ and $\Tfr$
have the same arity. When $\Recr$ involves only syntax trees of degree
two, $\Recr$ is {\em quadratic}. We say that a syntax tree $\Sfr'$ can
be {\em rewritten} by $\Recr$ into $\Tfr'$ if there exist two syntax
trees $\Sfr$ and $\Tfr$ satisfying $\Sfr \Recr \Tfr$ and $\Sfr'$ has a
partial subtree equal to $\Sfr$ such that, by replacing it by $\Tfr$ in
$\Sfr'$, we obtain $\Tfr'$. By a slight but convenient abuse of notation,
we denote by $\Sfr' \Recr \Tfr'$ this property. When a syntax tree $\Tfr$
can be obtained by performing a sequence of $\Recr$-rewritings from a
syntax tree $\Sfr$, we say that $\Sfr$ is {\em rewritable} by $\Recr$
into $\Tfr$ and we denote this property by $\Sfr \overset{*}{\Recr} \Tfr$.
For instance, for $S := S(2) \sqcup S(3)$ with $S(2) := \{\La, \Lc\}$
and $S(3) := \{\Lb\}$, consider the rewrite rule $\Recr$ on $\OpLibre(S)$
satisfying
\begin{equation}
    \begin{split}
    \begin{tikzpicture}[xscale=.4,yscale=.25]
        \node(0)at(0.00,-2.00){};
        \node(2)at(1.00,-2.00){};
        \node(3)at(2.00,-2.00){};
        \node(1)at(1.00,0.00){\begin{math}\Lb\end{math}};
        \draw(0)--(1);
        \draw(2)--(1);
        \draw(3)--(1);
        \node(r)at(1.00,1.75){};
        \draw(r)--(1);
    \end{tikzpicture}
    \end{split}
    \begin{split} \enspace \Recr \enspace \end{split}
    \begin{split}
    \begin{tikzpicture}[xscale=.3,yscale=.3]
        \node(0)at(0.00,-3.33){};
        \node(2)at(2.00,-3.33){};
        \node(4)at(4.00,-1.67){};
        \node(1)at(1.00,-1.67){\begin{math}\La\end{math}};
        \node(3)at(3.00,0.00){\begin{math}\La\end{math}};
        \draw(0)--(1);
        \draw(1)--(3);
        \draw(2)--(1);
        \draw(4)--(3);
        \node(r)at(3.00,1.5){};
        \draw(r)--(3);
    \end{tikzpicture}
    \end{split}
    \qquad \mbox{and} \qquad
    \begin{split}
    \begin{tikzpicture}[xscale=.3,yscale=.3]
        \node(0)at(0.00,-3.33){};
        \node(2)at(2.00,-3.33){};
        \node(4)at(4.00,-1.67){};
        \node(1)at(1.00,-1.67){\begin{math}\La\end{math}};
        \node(3)at(3.00,0.00){\begin{math}\Lc\end{math}};
        \draw(0)--(1);
        \draw(1)--(3);
        \draw(2)--(1);
        \draw(4)--(3);
        \node(r)at(3.00,1.5){};
        \draw(r)--(3);
    \end{tikzpicture}
    \end{split}
    \begin{split} \enspace \Recr \enspace \end{split}
    \begin{split}
    \begin{tikzpicture}[xscale=.3,yscale=.3]
        \node(0)at(0.00,-1.67){};
        \node(2)at(2.00,-3.33){};
        \node(4)at(4.00,-3.33){};
        \node(1)at(1.00,0.00){\begin{math}\La\end{math}};
        \node(3)at(3.00,-1.67){\begin{math}\Lc\end{math}};
        \draw(0)--(1);
        \draw(2)--(3);
        \draw(3)--(1);
        \draw(4)--(3);
        \node(r)at(1.00,1.5){};
        \draw(r)--(1);
    \end{tikzpicture}
    \end{split}\,.
\end{equation}
We then have the following sequence of rewritings
\begin{equation}
    \begin{split}
    \begin{tikzpicture}[xscale=.22,yscale=.17]
        \node(0)at(0.00,-6.50){};
        \node(10)at(8.00,-6.50){};
        \node(12)at(10.00,-6.50){};
        \node(2)at(1.00,-9.75){};
        \node(4)at(3.00,-9.75){};
        \node(5)at(4.00,-9.75){};
        \node(7)at(5.00,-9.75){};
        \node(8)at(6.00,-9.75){};
        \node[text=Rouge](1)at(2.00,-3.25){\begin{math}\Lb\end{math}};
        \node(11)at(9.00,-3.25){\begin{math}\La\end{math}};
        \node(3)at(2.00,-6.50){\begin{math}\Lc\end{math}};
        \node(6)at(5.00,-6.50){\begin{math}\Lb\end{math}};
        \node(9)at(7.00,0.00){\begin{math}\Lc\end{math}};
        \draw[draw=Rouge](0)--(1);
        \draw[draw=Rouge](1)--(9);
        \draw(10)--(11);
        \draw(11)--(9);
        \draw(12)--(11);
        \draw(2)--(3);
        \draw[draw=Rouge](3)--(1);
        \draw(4)--(3);
        \draw(5)--(6);
        \draw[draw=Rouge](6)--(1);
        \draw(7)--(6);
        \draw(8)--(6);
        \node(r)at(7.00,2.5){};
        \draw(r)--(9);
    \end{tikzpicture}
    \end{split}
    \begin{split} \enspace \Recr \enspace \end{split}
    \begin{split}
    \begin{tikzpicture}[xscale=.22,yscale=.17]
        \node(0)at(0.00,-8.40){};
        \node(11)at(10.00,-5.60){};
        \node(13)at(12.00,-5.60){};
        \node(2)at(2.00,-11.20){};
        \node(4)at(4.00,-11.20){};
        \node(6)at(6.00,-8.40){};
        \node(8)at(7.00,-8.40){};
        \node(9)at(8.00,-8.40){};
        \node(1)at(1.00,-5.60){\begin{math}\La\end{math}};
        \node[text=Rouge](10)at(9.00,0.00){\begin{math}\Lc\end{math}};
        \node(12)at(11.00,-2.80){\begin{math}\La\end{math}};
        \node(3)at(3.00,-8.40){\begin{math}\Lc\end{math}};
        \node[text=Rouge](5)at(5.00,-2.80){\begin{math}\La\end{math}};
        \node(7)at(7.00,-5.60){\begin{math}\Lb\end{math}};
        \draw(0)--(1);
        \draw[draw=Rouge](1)--(5);
        \draw(11)--(12);
        \draw[draw=Rouge](12)--(10);
        \draw(13)--(12);
        \draw(2)--(3);
        \draw(3)--(1);
        \draw(4)--(3);
        \draw[draw=Rouge](5)--(10);
        \draw(6)--(7);
        \draw[draw=Rouge](7)--(5);
        \draw(8)--(7);
        \draw(9)--(7);
        \node(r)at(9.00,2.25){};
        \draw[draw=Rouge](r)--(10);
    \end{tikzpicture}
    \end{split}
    \begin{split} \enspace \Recr \enspace \end{split}
    \begin{split}
    \begin{tikzpicture}[xscale=.22,yscale=.17]
        \node(0)at(0.00,-7.00){};
        \node(11)at(10.00,-10.50){};
        \node(13)at(12.00,-10.50){};
        \node(2)at(2.00,-10.50){};
        \node(4)at(4.00,-10.50){};
        \node(6)at(6.00,-10.50){};
        \node(8)at(7.00,-10.50){};
        \node(9)at(8.00,-10.50){};
        \node(1)at(1.00,-3.50){\begin{math}\La\end{math}};
        \node(10)at(9.00,-3.50){\begin{math}\Lc\end{math}};
        \node(12)at(11.00,-7.00){\begin{math}\La\end{math}};
        \node(3)at(3.00,-7.00){\begin{math}\Lc\end{math}};
        \node(5)at(5.00,0.00){\begin{math}\La\end{math}};
        \node[text=Rouge](7)at(7.00,-7.00){\begin{math}\Lb\end{math}};
        \draw(0)--(1);
        \draw(1)--(5);
        \draw(10)--(5);
        \draw(11)--(12);
        \draw(12)--(10);
        \draw(13)--(12);
        \draw(2)--(3);
        \draw(3)--(1);
        \draw(4)--(3);
        \draw[draw=Rouge](6)--(7);
        \draw[draw=Rouge](7)--(10);
        \draw[draw=Rouge](8)--(7);
        \draw[draw=Rouge](9)--(7);
        \node(r)at(5.00,2.5){};
        \draw(r)--(5);
    \end{tikzpicture}
    \end{split}
    \begin{split} \enspace \Recr \enspace \end{split}
    \begin{split}
    \begin{tikzpicture}[xscale=.22,yscale=.17]
        \node(0)at(0.00,-6.00){};
        \node(10)at(10.00,-9.00){};
        \node(12)at(12.00,-9.00){};
        \node(14)at(14.00,-9.00){};
        \node(2)at(2.00,-9.00){};
        \node(4)at(4.00,-9.00){};
        \node(6)at(6.00,-12.00){};
        \node(8)at(8.00,-12.00){};
        \node(1)at(1.00,-3.00){\begin{math}\La\end{math}};
        \node(11)at(11.00,-3.00){\begin{math}\Lc\end{math}};
        \node(13)at(13.00,-6.00){\begin{math}\La\end{math}};
        \node(3)at(3.00,-6.00){\begin{math}\Lc\end{math}};
        \node(5)at(5.00,0.00){\begin{math}\La\end{math}};
        \node(7)at(7.00,-9.00){\begin{math}\La\end{math}};
        \node(9)at(9.00,-6.00){\begin{math}\La\end{math}};
        \draw(0)--(1);
        \draw(1)--(5);
        \draw(10)--(9);
        \draw(11)--(5);
        \draw(12)--(13);
        \draw(13)--(11);
        \draw(14)--(13);
        \draw(2)--(3);
        \draw(3)--(1);
        \draw(4)--(3);
        \draw(6)--(7);
        \draw(7)--(9);
        \draw(8)--(7);
        \draw(9)--(11);
        \node(r)at(5.00,2.25){};
        \draw(r)--(5);
    \end{tikzpicture}
    \end{split}\,.
\end{equation}
We shall use the standard terminology ({\em confluent}, {\em terminating},
{\em convergent}, {\em normal form}, {\em critical pair}, {\em etc.})
about rewrite rules (see~\cite{BN98}).
\medskip

Any rewrite rule $\Recr$ on $\OpLibre(S)$ defines an operad congruence
$\Congr_\Recr$ on $\OpLibre(S)$ seen as a set-operad, the
{\em operad congruence induced} by $\Recr$, as the finest operad
congruence on $\OpLibre(S)$ containing the reflexive, symmetric, and
transitive closure of $\Recr$.
\medskip

\subsubsection{Koszulity}%
\label{subsubsec:Koszulite}
A quadratic operad $\Oca$ is {\em Koszul} if its Koszul complex is
acyclic~\cite{GK94,LV12}. In this work, to prove the Koszulity of an
operad $\Oca$, we shall make use of a combinatorial tool introduced by
Hoffbeck~\cite{Hof10} (see also~\cite{LV12}) consisting in exhibiting a
particular basis of $\Oca$, a so-called {\em Poincaré-Birkhoff-Witt basis}.
\medskip

In this paper, we shall use this tool only in the context of set-operads,
which reformulates, thanks to the work of Dotsenko and
Khoroshkin~\cite{DK10}, as follows. A set-operad $\Oca$ is Kosuzl if
there is a graded set $S$ and a rewrite rule $\Recr$ on $\OpLibre(S)$
such that $\Oca$ is isomorphic to $\OpLibre(S)/_{\Congr_\Recr}$ and
$\Recr$ is a convergent quadratic rewrite rule. Moreover, the set of
normal forms of $\Recr$ forms a Poincaré-Birkhoff-Witt basis of $\Oca$.
\medskip

\subsection{Diassociative operad}%
\label{subsec:dias}
We recall here, by using the notions presented during the previous
sections, the definition and some properties of the diassociative operad.
\medskip

The {\em diassociative operad} $\Dias$ was introduced by
Loday~\cite{Lod01} as the operad admitting the presentation
$\left(\GenLibre_{\Dias}, \RelLibre_{\Dias}\right)$ where
$\GenLibre_{\Dias} := \GenLibre_{\Dias}(2) := \{\GDias, \DDias\}$
and $\RelLibre_{\Dias}$ is the space induced by the equivalence
relation $\Congr$ satisfying
\begin{subequations}
\begin{equation} \label{equ:relation_dias_1}
    \GDias \circ_1 \DDias \; \Congr \;
    \DDias \circ_2 \GDias,
\end{equation}
\begin{equation} \label{equ:relation_dias_2}
    \GDias \circ_1 \GDias \; \Congr \;
    \GDias \circ_2 \GDias \; \Congr \;
    \GDias \circ_2 \DDias,
\end{equation}
\begin{equation} \label{equ:relation_dias_3}
    \DDias \circ_1 \GDias \; \Congr \;
    \DDias \circ_1 \DDias \; \Congr \;
    \DDias \circ_2 \DDias.
\end{equation}
\end{subequations}
Note that $\Dias$ is a binary and quadratic operad.
\medskip

This operad admits the following realization~\cite{Cha05}. For any
$n \geq 1$, $\Dias(n)$ is the linear span of the $\Efr_{n, k}$,
$k \in [n]$, and the partial compositions linearly satisfy, for all
$n, m \geq 1$, $k \in [n]$, $\ell \in [m]$, and $i \in [n]$,
\begin{equation}
    \Efr_{n, k} \circ_i \Efr_{m, \ell} =
    \begin{cases}
        \Efr_{n + m - 1, k + m - 1}
            & \mbox{if } i < k, \\
        \Efr_{n + m - 1, k + \ell - 1}
            & \mbox{if } i = k, \\
        \Efr_{n + m - 1, k}
            & \mbox{otherwise (} i > k \mbox{)}.
    \end{cases}
\end{equation}
Since the partial composition of two basis elements of $\Dias$ produces
exactly one basis element, $\Dias$ is well-defined as a set-operad.
Moreover, this realization shows that $\dim \Dias(n) = n$ and hence,
the Hilbert series of $\Dias$ satisfies
\begin{equation}
    \Hca_\Dias(t) = \frac{t}{(1 - t)^2}.
\end{equation}
\medskip

From the presentation of $\Dias$, we deduce that any $\Dias$-algebra,
also called {\em diassociative algebra}, is a vector space $\Alg_\Dias$
endowed with linear operations $\GDias$ and $\DDias$ satisfying the
relations encoded by~\eqref{equ:relation_dias_1}---%
\eqref{equ:relation_dias_3}.
\medskip

From the realization of $\Dias$, we deduce that the free diassociative
algebra $\AlgLibre_\Dias$ over one generator is the vector space $\Dias$
endowed with the linear operations
\begin{equation}
    \GDias, \DDias :
    \AlgLibre_\Dias \otimes \AlgLibre_\Dias \to \AlgLibre_\Dias,
\end{equation}
satisfying, for all $n, m \geq 1$, $k \in [n]$, $\ell \in [m]$,
\begin{equation}
    \Efr_{n, k} \GDias \Efr_{m, \ell}
        = (\Efr_{n, k} \otimes \Efr_{m, \ell}) \Action \Efr_{2, 1}
        = (\Efr_{2, 1} \circ_2 \Efr_{m, \ell}) \circ_1 \Efr_{n, k}
        = \Efr_{n + m, k},
\end{equation}
and
\begin{equation}
    \Efr_{n, k} \DDias \Efr_{m, \ell}
        = (\Efr_{n, k} \otimes \Efr_{m, \ell}) \Action \Efr_{2, 2}
        = (\Efr_{2, 2} \circ_2 \Efr_{m, \ell}) \circ_1 \Efr_{n, k}
        = \Efr_{n + m, n + \ell}.
\end{equation}
\medskip

As shown in~\cite{Gir12,Gir15}, the diassociative operad is isomorphic
to the suboperad of $\T \Mca$ generated by $01$ and $10$ where $\Mca$ is
the multiplicative monoid on $\{0, 1\}$. The concerned isomorphism sends
any $\Efr_{n, k}$ of $\Dias$ to the word $0^{k - 1} \, 1 \, 0^{n - k}$ of
$\T \Mca$.
\medskip

\section{Pluriassociative operads}%
\label{sec:dias_gamma}
In this section, we define the main object of this work: a generalization
on a nonnegative integer parameter $\gamma$ of the diassociative operad.
We provide a complete study of this new operad.
\medskip

\subsection{Construction and first properties}
We define here our generalization of the diassociative operad using the
functor $\T$ (whose definition is recalled in
Section~\ref{subsec:monoides_vers_operades}). We then describe the
elements and establish the Hilbert series of our generalization.
\medskip

\subsubsection{Construction} \label{subsubsec:construction_dias_gamma}
For any integer $\gamma \geq 0$, let $\Mca_\gamma$ be the monoid
$\{0\} \cup [\gamma]$ with the binary operation $\max$ as product,
denoted by $\Max$. We define $\Dias_\gamma$ as the suboperad of
$\T \Mca_\gamma$ generated by
\begin{equation} \label{equ:generateurs_dias_gamma}
    \{0a, a0 : a \in [\gamma]\}.
\end{equation}
By definition, $\Dias_\gamma$ is the vector space of words that can be
obtained by partial compositions of words
of~\eqref{equ:generateurs_dias_gamma}. We have, for instance,
\begin{equation}
    \Dias_2(1)
    =\Vect(\{0\}),
\end{equation}
\begin{equation}
    \Dias_2(2)
    =\Vect(\{01, 02, 10, 20\}),
\end{equation}
\begin{equation}
    \Dias_2(3)
    =\Vect(\{011, 012, 021, 022, 101, 102, 201, 202, 110, 120, 210, 220\}),
\end{equation}
and
\begin{equation}
    \textcolor{Bleu}{211} {\bf 2} \textcolor{Bleu}{01}
        \circ_4 \textcolor{Rouge}{31103}
    = \textcolor{Bleu}{211} \textcolor{Rouge}{32223} \textcolor{Bleu}{01},
\end{equation}
\begin{equation}
    \textcolor{Bleu}{11} {\bf 1} \textcolor{Bleu}{101}
        \circ_3 \textcolor{Rouge}{20}
    = \textcolor{Bleu}{11} \textcolor{Rouge}{21} \textcolor{Bleu}{101},
\end{equation}
\begin{equation}
    \textcolor{Bleu}{1} {\bf 0} \textcolor{Bleu}{13}
        \circ_2 \textcolor{Rouge}{210}
    = \textcolor{Bleu}{1} \textcolor{Rouge}{210} \textcolor{Bleu}{13}.
\end{equation}
\medskip

It follows immediately from the definition of $\Dias_\gamma$ as a
suboperad of $\T \Mca_\gamma$ that $\Dias_\gamma$ is a set-operad. Indeed,
any partial composition of two basis elements of $\Dias_\gamma$ gives
rises to exactly one basis element. We then shall see $\Dias_\gamma$ as
a set-operad over all Section~\ref{sec:dias_gamma}.
\medskip

Notice that $\Dias_\gamma(2)$ is the set~\eqref{equ:generateurs_dias_gamma}
of generators of $\Dias_\gamma$. Besides, observe that $\Dias_0$ is the
trivial operad and that $\Dias_\gamma$ is a suboperad of $\Dias_{\gamma + 1}$.
We call $\Dias_\gamma$ the {\em $\gamma$-pluriassociative operad}.
\medskip

\subsubsection{Elements and dimensions}
\begin{Proposition} \label{prop:elements_dias_gamma}
    For any integer $\gamma \geq 0$, as a set-operad, the underlying set
    of $\Dias_\gamma$ is the set of the words on the alphabet
    $\{0\} \cup [\gamma]$ containing exactly one occurrence of $0$.
\end{Proposition}
\begin{proof}
    Let us show that any word $x$ of $\Dias_\gamma$ satisfies the
    statement of the proposition by induction on the length $n$ of $x$.
    This is true when $n = 1$ because we necessarily have $x = 0$.
    Otherwise, when $n \geq 2$, there is a word $y$ of $\Dias_\gamma$
    of length $n - 1$ and a generator $g$ of $\Dias_\gamma$ such that
    $x = y \circ_i g$ for a $i \in [n - 1]$. Then, $x$ is obtained by
    replacing the $i$th letter $a$ of $y$ by the factor $u := u_1 u_2$
    where $u_1 := a \Max g_1$ and $u_2 := a \Max g_2$. Since $g$ contains
    exactly one $0$, this operation consists in inserting a nonzero
    letter of $[\gamma]$ into $y$. Since by induction hypothesis $y$
    contains exactly one $0$, it follows that $x$ satisfies the
    statement of the proposition.
    \smallskip

    Conversely, let us show that any word $x$ satisfying the statement
    of the proposition belongs to $\Dias_\gamma$ by induction on the
    length $n$ of $x$. This is true when $n = 1$ because we necessarily
    have $x = 0$ and $0$ belongs to $\Dias_\gamma$ since it is its unit.
    Otherwise, when $n \geq 2$, there is an integer $i \in [n - 1]$ such
    that $x_i x_{i + 1} \in \{0a, a0\}$ for an $a \in [\gamma]$. Let us
    suppose without loss of generality that $x_i x_{i + 1} = a0$. By
    setting $y$ as the word obtained by erasing the $i$th letter of $x$,
    we have $x = y \circ_i a0$. Thus, since by induction hypothesis $y$
    is an element of $\Dias_\gamma$, it follows that $x$ also is.
\end{proof}
\medskip

We deduce from Proposition~\ref{prop:elements_dias_gamma} that the
Hilbert series of $\Dias_\gamma$ satisfies
\begin{equation} \label{equ:serie_hilbert_dias_gamma}
    \Hca_{\Dias_\gamma}(t) = \frac{t}{(1 - \gamma t)^2}
\end{equation}
and that for all $n \geq 1$, $\dim \Dias_\gamma(n) = n \gamma^{n - 1}$.
For instance, the first dimensions of $\Dias_1$, $\Dias_2$, $\Dias_3$,
and $\Dias_4$ are respectively
\begin{equation}
    1, 2, 3, 4, 5, 6, 7, 8, 9, 10, 11,
\end{equation}
\begin{equation}
    1, 4, 12, 32, 80, 192, 448, 1024, 2304, 5120, 11264,
\end{equation}
\begin{equation}
    1, 6, 27, 108, 405, 1458, 5103, 17496, 59049, 196830, 649539,
\end{equation}
\begin{equation}
    1, 8, 48, 256, 1280, 6144, 28672, 131072, 589824, 2621440, 11534336.
\end{equation}
The second one is Sequence~\Sloane{A001787}, the third one is
Sequence~\Sloane{A027471}, and the last one is Sequence~\Sloane{A002697}
of~\cite{Slo}.
\medskip

\subsection{Presentation by generators and relations}%
\label{subsec:presentation_dias_gamma}
To establish a presentation of $\Dias_\gamma$, we shall start by defining
a morphism $\Mot_\gamma$ from a free operad to $\Dias_\gamma$. Then,
after showing that $\Mot_\gamma$ is a surjection, we will show that
$\Mot_\gamma$ induces an operad isomorphism between a quotient of
a free operad by a certain operad congruence $\Congr_\gamma$ and
$\Dias_\gamma$. The space of relations of $\Dias_\gamma$ of its
presentation will be induced by $\Congr_\gamma$.
\medskip

\subsubsection{From syntax trees to words}%
\label{subsubsec:arbres_vers_mots}
For any integer $\gamma \geq 0$, let $\GenDias := \GenDias(2)$ be the
graded set where
\begin{equation}
    \GenDias(2) := \{\GDias_a, \DDias_a : a \in [\gamma]\}.
\end{equation}
\medskip

Let $\Tfr$ be a syntax tree of $\OpLibre\left(\GenDias\right)$ and $x$
be a leaf of $\Tfr$. We say that an integer $a \in \{0\} \cup [\gamma]$
is {\em eligible} for $x$ if $a = 0$ or there is an ancestor $y$ of $x$
labeled by $\GDias_a$ (resp. $\DDias_a$) and $x$ is in the right (resp.
left) subtree of $y$. The {\em image} of $x$ is its greatest eligible
integer. Moreover, let
\begin{equation} \label{equ:application_mot_gamma}
    \Mot_\gamma : \OpLibre\left(\GenDias\right)(n) \to \Dias_\gamma(n),
    \qquad n \geq 1,
\end{equation}
the map where $\Mot_\gamma(\Tfr)$ is the word obtained by considering,
from left to right, the images of the leaves of $\Tfr$ (see
Figure~\ref{fig:exemple_mot_gamma}).
\begin{figure}[ht]
    \centering
     \begin{tikzpicture}[xscale=.28,yscale=.18]
        \node(0)at(0.00,-11.50){};
        \node(10)at(10.00,-7.67){};
        \node(12)at(12.00,-15.33){};
        \node(14)at(14.00,-15.33){};
        \node(16)at(16.00,-19.17){};
        \node(18)at(18.00,-19.17){};
        \node(2)at(2.00,-11.50){};
        \node(20)at(20.00,-19.17){};
        \node(22)at(22.00,-19.17){};
        \node(4)at(4.00,-11.50){};
        \node(6)at(6.00,-15.33){};
        \node(8)at(8.00,-15.33){};
        \node(1)at(1.00,-7.67){\begin{math}\GDias_4\end{math}};
        \node(11)at(11.00,-3.83){\begin{math}\DDias_2\end{math}};
        \node(13)at(13.00,-11.50){\begin{math}\DDias_1\end{math}};
        \node(15)at(15.00,-7.67){\begin{math}\DDias_3\end{math}};
        \node(17)at(17.00,-15.33){\begin{math}\DDias_2\end{math}};
        \node(19)at(19.00,-11.50){\begin{math}\GDias_1\end{math}};
        \node(21)at(21.00,-15.33){\begin{math}\DDias_4\end{math}};
        \node(3)at(3.00,-3.83){\begin{math}\DDias_3\end{math}};
        \node(5)at(5.00,-7.67){\begin{math}\GDias_1\end{math}};
        \node(7)at(7.00,-11.50){\begin{math}\GDias_2\end{math}};
        \node(9)at(9.00,0.00){\begin{math}\GDias_2\end{math}};
        \draw(0)--(1); \draw(1)--(3); \draw(10)--(11); \draw(11)--(9);
        \draw(12)--(13); \draw(13)--(15); \draw(14)--(13); \draw(15)--(11);
        \draw(16)--(17); \draw(17)--(19); \draw(18)--(17); \draw(19)--(15);
        \draw(2)--(1); \draw(20)--(21); \draw(21)--(19); \draw(22)--(21);
        \draw(3)--(9); \draw(4)--(5); \draw(5)--(3); \draw(6)--(7);
        \draw(7)--(5); \draw(8)--(7);
        \node(r)at(9,3){};
        \draw(9)--(r);
        \node[below of=0,node distance=3mm]
            {\small \begin{math}\textcolor{Bleu}{3}\end{math}};
        \node[below of=2,node distance=3mm]
            {\small \begin{math}\textcolor{Bleu}{4}\end{math}};
        \node[below of=4,node distance=3mm]
            {\small \begin{math}\textcolor{Bleu}{0}\end{math}};
        \node[below of=6,node distance=3mm]
            {\small \begin{math}\textcolor{Bleu}{1}\end{math}};
        \node[below of=8,node distance=3mm]
            {\small \begin{math}\textcolor{Bleu}{2}\end{math}};
        \node[below of=10,node distance=3mm]
            {\small \begin{math}\textcolor{Bleu}{2}\end{math}};
        \node[below of=12,node distance=3mm]
            {\small \begin{math}\textcolor{Bleu}{3}\end{math}};
        \node[below of=14,node distance=3mm]
            {\small \begin{math}\textcolor{Bleu}{3}\end{math}};
        \node[below of=16,node distance=3mm]
            {\small \begin{math}\textcolor{Bleu}{2}\end{math}};
        \node[below of=18,node distance=3mm]
            {\small \begin{math}\textcolor{Bleu}{2}\end{math}};
        \node[below of=20,node distance=3mm]
            {\small \begin{math}\textcolor{Bleu}{4}\end{math}};
        \node[below of=22,node distance=3mm]
            {\small \begin{math}\textcolor{Bleu}{2}\end{math}};
    \end{tikzpicture}
    \caption{A syntax tree $\Tfr$ of $\OpLibre\left(\GenDias\right)$
    where images of its leaves are shown. This tree satisfies
    $\Mot_\gamma(\Tfr) = \textcolor{Bleu}{340122332242}$.}
    \label{fig:exemple_mot_gamma}
\end{figure}
\medskip

\begin{Lemme} \label{lem:mot_gamma_morphisme}
    For any integer $\gamma \geq 0$, the map $\Mot_\gamma$ is an
    operad morphism from $\OpLibre\left(\GenDias\right)$ to
    $\Dias_\gamma$.
\end{Lemme}
\begin{proof}
    Let us first show that $\Mot_\gamma$ is a well-defined map. Let
    $\Tfr$ be a syntax tree of $\OpLibre\left(\GenDias\right)$ of arity
    $n$. Observe that by starting from the root of $\Tfr$, there is a
    unique maximal path obtained by following the directions specified
    by its internal nodes (a $\GDias_a$ means to go the left child while
    a $\DDias_a$ means to go to the right child). Then, the leaf at the
    end of this path is the only leaf with $0$ as image. Others $n - 1$
    leaves have integers of $[\gamma]$ as images. By
    Proposition~\ref{prop:elements_dias_gamma}, this implies that
    $\Mot_\gamma(\Tfr)$ is an element of $\Dias_\gamma(n)$.
    \smallskip

    To prove that $\Mot_\gamma$ is an operad morphism, we consider its
    following alternative description. If $\Tfr$ is a syntax tree of
    $\OpLibre\left(\GenDias\right)$, we can consider the tree $\Tfr'$
    obtained by replacing in $\Tfr$ each label $\GDias_a$ (resp. $\DDias_a$)
    by the word $0a$ (resp. $a0$), where $a \in [\gamma]$. Then, by a
    straightforward induction on the number of internal nodes of $\Tfr$,
    we obtain that $\Eval_{\Dias_\gamma}(\Tfr')$, where $\Tfr'$ is seen
    as a syntax tree of $\OpLibre\left(\Dias_\gamma(2)\right)$, is
    $\Mot_\gamma(\Tfr)$. It then follows that $\Mot_\gamma$ is an operad
    morphism.
\end{proof}
\medskip

\subsubsection{Hook syntax trees}
Let us now consider the map
\begin{equation} \label{equ:application_equerre_gamma}
    \Equerre_\gamma : \Dias_\gamma(n) \to \OpLibre\left(\GenDias\right)(n),
    \qquad n \geq 1,
\end{equation}
defined for any word $x$ of $\Dias_\gamma$ by
\begin{equation} \label{equ:definition_application_equerre_gamma}
    \begin{split}\Equerre_\gamma(x)\end{split} :=
    \begin{split}
    \begin{tikzpicture}[xscale=.5,yscale=.45]
        \node(0)at(0.00,-5.40){};
        \node(2)at(3.00,-7.20){};
        \node(4)at(5.00,-7.20){};
        \node(6)at(6.00,-3.60){};
        \node(8)at(9.00,-1.80){};
        \node(1)at(1.00,-3.60){\begin{math}\DDias_{u_1}\end{math}};
        \node(3)at(4.00,-5.40){\begin{math}\DDias_{u_{|u|}}\end{math}};
        \node(5)at(5.00,-1.80){\begin{math}\GDias_{v_1}\end{math}};
        \node(7)at(8.00,0.00){\begin{math}\GDias_{v_{|v|}}\end{math}};
        \node(r)at(8,1.5){};
        \draw(0)--(1);
        \draw(1)--(5);
        \draw(2)--(3);
        \draw[densely dashed](3)--(1);
        \draw(4)--(3);
        \draw[densely dashed](5)--(7);
        \draw(6)--(5);
        \draw(8)--(7);
        \draw(7)--(r);
    \end{tikzpicture}
    \end{split}\,,
\end{equation}
where $x$ decomposes, by Proposition~\ref{prop:elements_dias_gamma},
uniquely in $x = u0v$ where $u$ and $v$ are words on the alphabet
$[\gamma]$. The dashed edges denote, depending on their orientation,
a right comb (wherein internal nodes are labeled, from top to bottom
by $\DDias_{u_1}$, \dots, $\DDias_{u_{|u|}}$) or a left comb
(wherein internal nodes are labeled, from bottom to top, by
$\GDias_{v_1}$, \dots, $\GDias_{v_{|v|}}$). We shall call any syntax
tree of the form \eqref{equ:definition_application_equerre_gamma} a
{\em hook syntax tree}.
\medskip

\begin{Lemme} \label{lem:mot_gamma_surjection}
    For any integer $\gamma \geq 0$, the map $\Mot_\gamma$ is a
    surjective operad morphism from $\OpLibre\left(\GenDias\right)$
    onto $\Dias_\gamma$. Moreover, for any element $x$ of $\Dias_\gamma$,
    $\Equerre_\gamma(x)$ belongs to the fiber of $x$ under $\Mot_\gamma$.
\end{Lemme}
\begin{proof}
    The fact that $x$ belongs to the fiber of $x$ under $\Mot_\gamma$ is
    an immediate consequence of the definitions of $\Mot_\gamma$ and
    $\Equerre_\gamma$, and the fact that by
    Proposition~\ref{prop:elements_dias_gamma}, any word $x$ of
    $\Dias_\gamma$ decomposes uniquely in $x = u0v$ where $u$ and $v$
    are words on the alphabet~$[\gamma]$. Then, $\Mot_\gamma$ is
    surjective as a map. Moreover, since by
    Lemma~\ref{lem:mot_gamma_morphisme}, $\Mot_\gamma$ is an operad
    morphism, it is a surjective operad morphism.
\end{proof}
\medskip

\subsubsection{A rewrite rule on syntax trees}
Let $\Recr_\gamma$ be the quadratic rewrite rule on
$\OpLibre\left(\GenDias\right)$ satisfying
\begin{subequations}
\begin{equation} \label{equ:reecriture_dias_gamma_1}
    \DDias_{a'} \circ_2 \GDias_a
    \enspace \Recr_\gamma \enspace
    \GDias_a \circ_1 \DDias_{a'},
    \qquad a, a' \in [\gamma],
\end{equation}
\begin{equation} \label{equ:reecriture_dias_gamma_2}
    \GDias_a \circ_2 \DDias_b
    \enspace \Recr_\gamma \enspace
    \GDias_a \circ_1 \GDias_b,
    \qquad a < b \in [\gamma],
\end{equation}
\begin{equation} \label{equ:reecriture_dias_gamma_3}
    \DDias_a \circ_1 \GDias_b
    \enspace \Recr_\gamma \enspace
    \DDias_a \circ_2 \DDias_b,
    \qquad a < b \in [\gamma],
\end{equation}
\begin{equation} \label{equ:reecriture_dias_gamma_4}
    \GDias_a \circ_2 \GDias_b
    \enspace \Recr_\gamma \enspace
    \GDias_b \circ_1 \GDias_a,
    \qquad a < b \in [\gamma],
\end{equation}
\begin{equation} \label{equ:reecriture_dias_gamma_5}
    \DDias_a \circ_1 \DDias_b
    \enspace \Recr_\gamma \enspace
    \DDias_b \circ_2 \DDias_a,
    \qquad a < b \in [\gamma],
\end{equation}
\begin{equation} \label{equ:reecriture_dias_gamma_6}
    \GDias_d \circ_2 \GDias_c
    \enspace \Recr_\gamma \enspace
    \GDias_d \circ_1 \GDias_d,
    \qquad c \leq d \in [\gamma],
\end{equation}
\begin{equation} \label{equ:reecriture_dias_gamma_7}
    \GDias_d \circ_2 \DDias_c
    \enspace \Recr_\gamma \enspace
    \GDias_d \circ_1 \GDias_d,
    \qquad c \leq d \in [\gamma],
\end{equation}
\begin{equation} \label{equ:reecriture_dias_gamma_8}
    \DDias_d \circ_1 \GDias_c
    \enspace \Recr_\gamma \enspace
    \DDias_d \circ_2 \DDias_d,
    \qquad c \leq d \in [\gamma],
\end{equation}
\begin{equation} \label{equ:reecriture_dias_gamma_9}
    \DDias_d \circ_1 \DDias_c
    \enspace \Recr_\gamma \enspace
    \DDias_d \circ_2 \DDias_d,
    \qquad c \leq d \in [\gamma],
\end{equation}
\end{subequations}
and denote by $\Congr_\gamma$ the operadic congruence on
$\OpLibre\left(\GenDias\right)$ induced by $\Recr_\gamma$.
\medskip

\begin{Lemme} \label{lem:mot_gamma_stable_classes_equivalence}
    For any integer $\gamma \geq 0$ and any syntax trees $\Tfr_1$ and
    $\Tfr_2$ of $\OpLibre\left(\GenDias\right)$,
    $\Tfr_1 \Congr_\gamma \Tfr_2$ implies
    $\Mot_\gamma(\Tfr_1) = \Mot_\gamma(\Tfr_2)$.
\end{Lemme}
\begin{proof}
    Let us denote by $\Rel_\gamma$ the symmetric closure of $\Recr_\gamma$.
    In the first place, observe that for any relation
    $\Sfr_1 \Rel_\gamma \Sfr_2$ where $\Sfr_1$ and $\Sfr_2$ are syntax
    trees of $\OpLibre\left(\GenDias\right)(3)$, for any $i \in [3]$,
    the eligible integers for the $i$th leaves of $\Sfr_1$ and $\Sfr_2$
    are the same. Besides, by definition of $\Congr_\gamma$, since
    $\Tfr_1 \Congr_\gamma \Tfr_2$, one can obtain $\Tfr_2$ from $\Tfr_1$
    by performing a sequence of $\Rel_\gamma$-rewritings. According to
    the previous observation, a $\Rel_\gamma$-rewriting preserve the
    eligible integers of all leaves of the tree on which they are
    performed. Therefore, the images of the leaves of $\Tfr_2$ are, from
    left to right, the same as the images of the leaves of $\Tfr_1$ and
    hence, $\Mot_\gamma(\Tfr_1) = \Mot_\gamma(\Tfr_2)$.
\end{proof}
\medskip

Lemma~\ref{lem:mot_gamma_stable_classes_equivalence} implies that the map
\begin{equation}
    \bar\Mot_\gamma :
    \OpLibre\left(\GenDias\right)(n)/_{\Congr_\gamma}
    \to \Dias_\gamma(n),
    \qquad n \geq 1,
\end{equation}
satisfying, for any $\Congr_\gamma$-equivalence class
$[\Tfr]_{\Congr_\gamma}$,
\begin{equation}
    \bar\Mot_\gamma\left([\Tfr]_\gamma\right) = \Mot_\gamma(\Tfr),
\end{equation}
where $\Tfr$ is any tree of $[\Tfr]_{\Congr_\gamma}$ is well-defined.
\medskip

\begin{Lemme} \label{lem:reecriture_dias_gamma}
    For any integer $\gamma \geq 0$, any syntax tree $\Tfr$ of
    $\OpLibre\left(\GenDias\right)$ can be rewritten, by a sequence of
    $\Recr_\gamma$-rewritings, into a hook syntax tree. Moreover, this
    hook syntax tree is $\Equerre_\gamma(\Mot_\gamma(\Tfr))$.
\end{Lemme}
\begin{proof}
    In the following, to gain readability, we shall denote by $\GDias_*$
    (resp. $\DDias_*$) any element $\GDias_a$ (resp. $\DDias_a$) of
    $\GenDias$ when taking into account the value of $a \in [\gamma]$
    is not necessary. Using this notation,
    from~\eqref{equ:reecriture_dias_gamma_1}---%
    \eqref{equ:reecriture_dias_gamma_9}, we observe that
    $\Recr_\gamma$
    expresses as
    \begin{subequations}
    \begin{equation} \label{equ:reecriture_sans_etiq_dias_gamma_1}
        \DDias_* \circ_2 \GDias_*
        \enspace \Recr_\gamma \enspace
        \GDias_* \circ_1 \DDias_*,
    \end{equation}
    \begin{equation} \label{equ:reecriture_sans_etiq_dias_gamma_2}
        \GDias_* \circ_2 \DDias_*
        \enspace \Recr_\gamma \enspace
        \GDias_* \circ_1 \GDias_*,
    \end{equation}
    \begin{equation} \label{equ:reecriture_sans_etiq_dias_gamma_3}
        \DDias_* \circ_1 \GDias_*
        \enspace \Recr_\gamma \enspace
        \DDias_* \circ_2 \DDias_*,
    \end{equation}
    \begin{equation} \label{equ:reecriture_sans_etiq_dias_gamma_4}
        \GDias_* \circ_2 \GDias_*
        \enspace \Recr_\gamma \enspace
        \GDias_* \circ_1 \GDias_*,
    \end{equation}
    \begin{equation} \label{equ:reecriture_sans_etiq_dias_gamma_5}
        \DDias_* \circ_1 \DDias_*
        \enspace \Recr_\gamma \enspace
        \DDias_* \circ_2 \DDias_*.
    \end{equation}
    \end{subequations}
    \smallskip

    Let us first focus on the first part of the statement of the lemma
    to show that $\Tfr$ is rewritable by $\Recr_\gamma$ into a hook
    syntax tree. We reason by induction on the arity $n$ of $\Tfr$. When
    $n \leq 2$, $\Tfr$ is immediately a hook syntax tree. Otherwise,
    $\Tfr$ has at least two internal nodes. Then, $\Tfr$ is made of a
    root connected to a first subtree $\Tfr_1$ and a second subtree
    $\Tfr_2$. By induction hypothesis, $\Tfr$ is rewritable by
    $\Recr_\gamma$ into a tree made of a root $r$ of the same label as
    the one of the root of $\Tfr$, connected to a first subtree $\Sfr_1$
    such that $\Tfr_1 \overset{*}{\Recr_\gamma} \Sfr_1$ and a second
    subtree $\Sfr_2$ such that $\Tfr_2 \overset{*}{\Recr_\gamma} \Sfr_2$,
    both being hook syntax trees. We have to deal two cases following
    the number of internal nodes of $\Tfr_1$.
    \smallskip

    \begin{enumerate}[label={\it Case \arabic*.},fullwidth]
        \item \label{item:reecriture_dias_gamma_cas_1}
        If $\Tfr_1$ has at least one internal node, we have the two
        $\overset{*}{\Recr_\gamma}$-relations
        \begin{equation} \label{equ:reecriture_dias_gamma_cas_1}
            \begin{split}
                \Tfr \enspace \overset{*}{\Recr_\gamma} \enspace
            \end{split}
            \begin{split}
            \begin{tikzpicture}[xscale=.4,yscale=.3]
                \node(0)at(0.00,-7.43){};
                \node(10)at(10.00,-3.71){};
                \node(12)at(12.00,-1.86){\begin{math}\Sfr_2\end{math}};
                \node(2)at(2.00,-9.29){};
                \node(4)at(4.00,-11.14){};
                \node(6)at(6.00,-11.14){};
                \node(8)at(8.00,-5.57){};
                \node(1)at(1.00,-5.57){\begin{math}x\end{math}};
                \node(11)at(11.00,0.00){\begin{math}r\end{math}};
                \node(3)at(3.00,-7.43){\begin{math}\DDias_*\end{math}};
                \node(5)at(5.00,-9.29){\begin{math}\DDias_*\end{math}};
                \node(7)at(7.00,-3.71){\begin{math}\GDias_*\end{math}};
                \node(9)at(9.00,-1.86){\begin{math}\GDias_*\end{math}};
                \draw(0)--(1); \draw(1)--(7); \draw(10)--(9);
                \draw(12)--(11); \draw(2)--(3); \draw(3)--(1);
                \draw(4)--(5); \draw[densely dashed](5)--(3);
                \draw(6)--(5); \draw[densely dashed](7)--(9);
                \draw(8)--(7); \draw(9)--(11);
                \node(r)at(11,1.5){}; \draw[](r)--(11);
            \end{tikzpicture}
            \end{split}
            \begin{split}
                \enspace \overset{*}{\Recr_\gamma} \enspace
            \end{split}
            \begin{split}
            \begin{tikzpicture}[xscale=.4,yscale=.3]
                \node(0)at(0.00,-8.57){};
                \node(10)at(10.00,-8.57){};
                \node(12)at(12.00,-4.29){};
                \node(14)at(14.00,-2.14){};
                \node(2)at(2.00,-10.71){};
                \node(4)at(4.00,-12.86){};
                \node(6)at(6.00,-12.86){};
                \node(8)at(8.00,-8.57){};
                \node(1)at(1.00,-6.43){\begin{math}x\end{math}};
                \node(11)at(11.00,-2.14){\begin{math}\GDias_*\end{math}};
                \node(13)at(13.00,0.00){\begin{math}\GDias_*\end{math}};
                \node(3)at(3.00,-8.57){\begin{math}\DDias_*\end{math}};
                \node(5)at(5.00,-10.71){\begin{math}\DDias_*\end{math}};
                \node(7)at(7.00,-4.29){\begin{math}\DDias_*\end{math}};
                \node(9)at(9.00,-6.43){\begin{math}\DDias_*\end{math}};
                \draw(0)--(1); \draw(1)--(7); \draw(10)--(9);
                \draw[densely dashed](11)--(13); \draw(12)--(11);
                \draw(14)--(13); \draw(2)--(3); \draw(3)--(1);
                \draw(4)--(5); \draw[densely dashed](5)--(3);
                \draw(6)--(5); \draw(7)--(11); \draw(8)--(9);
                \draw[densely dashed](9)--(7); \node(r)at(13,1.5){};
                \draw(r)--(13);
            \end{tikzpicture}
            \end{split}\,.
        \end{equation}
        The first $\overset{*}{\Recr_\gamma}$-relation
        of~\eqref{equ:reecriture_dias_gamma_cas_1} has just been
        explained. The second one comes from the application of the
        induction hypothesis on the upper part of the tree of the middle
        of~\eqref{equ:reecriture_dias_gamma_cas_1} obtained by cutting
        the edge connecting the node $x$ to its father. When the
        rightmost tree of~\eqref{equ:reecriture_dias_gamma_cas_1} is not
        already a hook syntax tree, one has two cases following the
        label of $x$.
        \smallskip

        \begin{enumerate}[label={\it Case \arabic{enumi}.\arabic*.},fullwidth]
            \item If $x$ is labeled by $\DDias_*$,
            by~\eqref{equ:reecriture_sans_etiq_dias_gamma_5}, the bottom
            part of the rightmost tree
            of~\eqref{equ:reecriture_dias_gamma_cas_1} consisting
            in internal nodes labeled by $\DDias_*$ is rewritable by
            $\Recr_\gamma$ into a right comb tree wherein internal nodes
            are labeled by $\DDias_*$. Then, the rightmost tree
            of~\eqref{equ:reecriture_dias_gamma_cas_1} is rewritable
            by $\Recr_\gamma$ into a hook syntax tree, and then $\Tfr$
            also is.
            \smallskip

            \item Otherwise, $x$ is labeled by $\GDias_*$. By definition
            of $\Equerre_\gamma$, the second subtree of $x$ is a leaf.
            By~\eqref{equ:reecriture_sans_etiq_dias_gamma_3}, the bottom
            part of the rightmost tree
            of~\eqref{equ:reecriture_dias_gamma_cas_1} consisting
            in $x$ and internal nodes labeled by $\DDias_*$ can be
            rewritten by $\Recr_\gamma$ into a right comb tree wherein
            internal nodes are labeled by $\DDias_*$. Then, the rightmost
            tree of \eqref{equ:reecriture_dias_gamma_cas_1} is
            rewritable by $\Recr_\gamma$ into a hook syntax tree, and
            then $\Tfr$ also is.
        \end{enumerate}
        \smallskip

        \item Otherwise, $\Tfr_1$ is the leaf. We then have the
        $\overset{*}{\Recr_\gamma}$-relation
        \begin{equation} \label{equ:reecriture_dias_gamma_cas_2}
            \begin{split}
                \Tfr \enspace \overset{*}{\Recr_\gamma} \enspace
            \end{split}
            \begin{split}
            \begin{tikzpicture}[xscale=.45,yscale=.4]
                \node(0)at(0.00,-1.67){};
                \node(2)at(2.00,-3.33){\begin{math}\Sfr_{21}\end{math}};
                \node(4)at(4.00,-3.33){\begin{math}\Sfr_{22}\end{math}};
                \node(1)at(1.00,0.00){\begin{math}r\end{math}};
                \node(3)at(3.00,-1.67){\begin{math}r'\end{math}};
                \draw(0)--(1); \draw(2)--(3); \draw(3)--(1); \draw(4)--(3);
                \node(r)at(1,1.25){};
                \draw(r)--(1);
            \end{tikzpicture}
            \end{split}\,,
        \end{equation}
        where $\Sfr_{21}$ is the first subtree of the root of $\Sfr_2$,
        $\Sfr_{22}$ is the second subtree of the root of $\Sfr_2$, and
        $r'$ is a node with the same label as the root of $\Sfr_2$.
        \smallskip

        \begin{enumerate}[label={\it Case \arabic{enumi}.\arabic*.},fullwidth]
            \item If $r \circ_2 r'$ is equal to
            $\DDias_* \circ_2 \GDias_*$, $\GDias_* \circ_2 \DDias_*$, or
            $\GDias_* \circ_2 \GDias_*$, respectively
            by~\eqref{equ:reecriture_sans_etiq_dias_gamma_1},
            \eqref{equ:reecriture_sans_etiq_dias_gamma_2},
            and~\eqref{equ:reecriture_sans_etiq_dias_gamma_4},
            the rightmost tree
            of~\eqref{equ:reecriture_dias_gamma_cas_2} can be rewritten
            by $\Recr_\gamma$ into a tree $\Rfr$ having a first subtree
            with at least one internal node. Hence, $\Rfr$ is of the
            form required to be treated
            by~\ref{item:reecriture_dias_gamma_cas_1}, implying that
            $\Tfr$ is rewritable by $\Recr_\gamma$ into a hook syntax tree.
            \smallskip

            \item Otherwise, $r \circ_2 r'$ is equal to
            $\DDias_* \circ_2 \DDias_*$. Since $\Sfr_2$ is by hypothesis
            a hook syntax tree, it is necessarily a right comb tree
            whose internal nodes are labeled by $\DDias_*$. Hence, the
            rightmost tree of~\eqref{equ:reecriture_dias_gamma_cas_2} is
            already a hook syntax tree, showing that $\Tfr$ is rewritable
            by $\Recr_\gamma$ into a hook syntax tree.
        \end{enumerate}
    \end{enumerate}
    \smallskip

    Let us finally show the last part of the statement of the lemma.
    Observe that, by definition of $\Equerre_\gamma$ and $\Mot_\gamma$,
    if $\Sfr_1$ and $\Sfr_2$ are two different hook syntax trees,
    $\Mot_\gamma(\Sfr_1) \ne \Mot_\gamma(\Sfr_2)$. We have just shown
    that $\Tfr$ is rewritable by $\Recr_\gamma$ into a hook syntax tree
    $\Sfr$. Besides, by Lemma~\ref{lem:mot_gamma_stable_classes_equivalence},
    one has $\Mot_\gamma(\Tfr) = \Mot_\gamma(\Sfr)$. Then, $\Sfr$ is
    necessarily the hook syntax tree $\Equerre_\gamma(\Mot_\gamma(\Tfr))$.
\end{proof}
\medskip

\subsubsection{Presentation by generators and relations}

\begin{Lemme} \label{lem:mot_gamma_quotient_bijection}
    For any integers $\gamma \geq 0$ and $n \geq 1$, the map
    $\bar\Mot_\gamma$ defines a bijection between
    $\OpLibre\left(\GenDias\right)(n)/_{\Congr_\gamma}$ and
    $\Dias_\gamma(n)$.
\end{Lemme}
\begin{proof}
    Let us show that $\bar\Mot_\gamma$ is injective. Let $\Tfr_1$ and
    $\Tfr_2$ be two syntax trees of $\OpLibre\left(\GenDias\right)$ such
    that $\Mot_\gamma(\Tfr_1) = \Mot_\gamma(\Tfr_2)$ and let
    $\Sfr := \Equerre_\gamma(\Mot_\gamma(\Tfr_1)) =
    \Equerre_\gamma(\Mot_\gamma(\Tfr_2))$.
    By Lemma~\ref{lem:reecriture_dias_gamma}, one has
    $\Tfr_1 \overset{*}{\Recr_\gamma} \Sfr$ and $
    \Tfr_2 \overset{*}{\Recr_\gamma} \Sfr$, and hence,
    $\Tfr_1 \Congr_\gamma \Tfr_2$. By the definition of the map
    $\bar \Mot_\gamma$ from the map $\Mot_\gamma$, this show that
    $\bar\Mot_\gamma$ is injective. Besides, by
    Lemma~\ref{lem:mot_gamma_surjection}, $\bar\Mot_\gamma$ is surjective,
    whence the statement of the lemma.
\end{proof}
\medskip

\begin{Theoreme} \label{thm:presentation_dias_gamma}
    For any integer $\gamma \geq 0$, the operad $\Dias_\gamma$
    admits the following presentation. It is generated by $\GenDias$ and
    its space of relations $\RelDias$ is the space induced by the
    equivalence relation $\Rel_\gamma$ satisfying
    \begin{subequations}
    \begin{equation} \label{equ:relation_dias_gamma_1}
        \GDias_a \circ_1 \DDias_{a'}
        \enspace \Rel_\gamma \enspace
        \DDias_{a'} \circ_2 \GDias_a,
        \qquad a, a' \in [\gamma],
    \end{equation}
    \begin{equation} \label{equ:relation_dias_gamma_2}
        \GDias_a \circ_1 \GDias_b
        \enspace \Rel_\gamma \enspace
        \GDias_a \circ_2 \DDias_b,
        \qquad a < b \in [\gamma],
    \end{equation}
    \begin{equation} \label{equ:relation_dias_gamma_3}
        \DDias_a \circ_1 \GDias_b
        \enspace \Rel_\gamma \enspace
        \DDias_a \circ_2 \DDias_b,
        \qquad a < b \in [\gamma],
    \end{equation}
    \begin{equation} \label{equ:relation_dias_gamma_4}
        \GDias_b \circ_1 \GDias_a
        \enspace \Rel_\gamma \enspace
        \GDias_a \circ_2 \GDias_b,
        \qquad a < b \in [\gamma],
    \end{equation}
    \begin{equation} \label{equ:relation_dias_gamma_5}
        \DDias_a \circ_1 \DDias_b
        \enspace \Rel_\gamma \enspace
        \DDias_b \circ_2 \DDias_a,
        \qquad a < b \in [\gamma],
    \end{equation}
    \begin{equation} \label{equ:relation_dias_gamma_6}
        \GDias_d \circ_1 \GDias_d
        \enspace \Rel_\gamma \enspace
        \GDias_d \circ_2 \GDias_c
        \enspace \Rel_\gamma \enspace
        \GDias_d \circ_2 \DDias_c,
        \qquad c \leq d \in [\gamma],
    \end{equation}
    \begin{equation} \label{equ:relation_dias_gamma_7}
        \DDias_d \circ_1 \GDias_c
        \enspace \Rel_\gamma \enspace
        \DDias_d \circ_1 \DDias_c
        \enspace \Rel_\gamma \enspace
        \DDias_d \circ_2 \DDias_d,
        \qquad c \leq d \in [\gamma].
    \end{equation}
    \end{subequations}
\end{Theoreme}
\begin{proof}
    By Lemma~\ref{lem:mot_gamma_quotient_bijection}, the map
    $\bar\Mot_\gamma$ is, for any $n \geq 1$, a bijection between the
    sets $\OpLibre\left(\GenDias\right)(n)/_{\Congr_\gamma}$ and
    $\Dias_\gamma(n)$. Moreover, by Lemma~\ref{lem:mot_gamma_morphisme},
    $\Mot_\gamma$ is an operad morphism, and then $\bar\Mot_\gamma$ also
    is. Hence, $\bar\Mot_\gamma$ is an operad isomorphism between
    $\OpLibre\left(\GenDias\right)/_{\Congr_\gamma}$ and $\Dias_\gamma$.
    Therefore, since $\RelLibre_{\Dias_\gamma}$ is the space induced by
    $\Congr_\gamma$, $\Dias_\gamma$ admits the stated presentation.
\end{proof}
\medskip

The space of relations $\RelDias$ of $\Dias_\gamma$ exhibited by
Theorem~\ref{thm:presentation_dias_gamma} can be rephrased in a more
compact way as the space generated by
\begin{subequations}
\begin{equation} \label{equ:relation_dias_gamma_1_concise}
    \GDias_a \circ_1 \DDias_{a'} - \DDias_{a'} \circ_2 \GDias_a,
    \qquad a, a' \in [\gamma],
\end{equation}
\begin{equation} \label{equ:relation_dias_gamma_2_concise}
    \GDias_a \circ_1 \GDias_{a \Max a'} - \GDias_a \circ_2 \DDias_{a'},
    \qquad a, a' \in [\gamma],
\end{equation}
\begin{equation} \label{equ:relation_dias_gamma_3_concise}
    \DDias_a \circ_1 \GDias_{a'} - \DDias_a \circ_2 \DDias_{a \Max a'},
    \qquad a, a' \in [\gamma],
\end{equation}
\begin{equation} \label{equ:relation_dias_gamma_4_concise}
    \GDias_{a \Max a'} \circ_1 \GDias_a - \GDias_a \circ_2 \GDias_{a'},
    \qquad a, a' \in [\gamma],
\end{equation}
\begin{equation} \label{equ:relation_dias_gamma_5_concise}
    \DDias_a \circ_1 \DDias_{a'} - \DDias_{a \Max a'} \circ_2 \DDias_a,
    \qquad a, a' \in [\gamma].
\end{equation}
\end{subequations}
\medskip

Observe that, by Theorem \ref{thm:presentation_dias_gamma}, $\Dias_1$
and the diassociative operad (see~\cite{Lod01} or
Section~\ref{subsec:dias}) admit the same presentation.
Then, for all integers $\gamma \geq 0$, the operads $\Dias_\gamma$
are generalizations of the diassociative operad.
\medskip

\subsection{Miscellaneous properties}
From the description of the elements of $\Dias_\gamma$ and its structure
revealed by its presentation, we develop here some of its properties.
Unless otherwise specified, $\Dias_\gamma$ is still considered in this
section as a set-operad.
\medskip

\subsubsection{Koszulity}

\begin{Theoreme} \label{thm:koszulite_dias_gamma}
    For any integer $\gamma \geq 0$, $\Dias_\gamma$ is a Koszul operad.
    Moreover, the set of hook syntax trees of $\OpLibre\left(\GenDias\right)$
    forms a Poincaré-Birkhoff-Witt basis of $\Dias_\gamma$.
\end{Theoreme}
\begin{proof}
    From the definition of hook syntax trees, it appears that no hook
    syntax tree can be rewritten by $\Recr_\gamma$ into another syntax
    tree. Hence, and by Lemma~\ref{lem:reecriture_dias_gamma},
    $\Recr_\gamma$ is a terminating rewrite rule and its normal forms
    are hook syntax trees. Moreover, again by
    Lemma~\ref{lem:reecriture_dias_gamma}, since any syntax tree is
    rewritable by $\Recr_\gamma$ into a unique hook syntax tree,
    $\Recr_\gamma$ is a confluent rewrite rule, and hence, $\Recr_\gamma$
    is convergent. Now, since by Theorem~\ref{thm:presentation_dias_gamma},
    the space of relations of $\Dias_\gamma$ is the space induced by the
    operad congruence induced by $\Recr_\gamma$, by the Koszulity
    criterion~\cite{Hof10,DK10,LV12} we have reformulated in
    Section~\ref{subsubsec:Koszulite}, $\Dias_\gamma$ is a Koszul
    operad and the set of of hook syntax trees of
    $\OpLibre\left(\GenDias\right)$ forms a Poincaré-Birkhoff-Witt basis
    of $\Dias_\gamma$.
\end{proof}
\medskip

\subsubsection{Symmetries}
If $\Oca_1$ and $\Oca_2$ are two operads, a linear map
$\phi : \Oca_1 \to \Oca_2$ is an {\em operad antimorphism} if it
respects arities and anticommutes with partial composition maps, that is,
\begin{equation}
    \phi(x \circ_i y) = \phi(x) \circ_{n - i + 1} \phi(y),
    \qquad x \in \Oca(n), y \in \Oca, i \in [n].
\end{equation}
A {\em symmetry} of an operad $\Oca$ is either an automorphism or an
antiautomorphism. The set of all symmetries of $\Oca$ form a group for
the composition, called the {\em group of symmetries} of $\Oca$.
\medskip

\begin{Proposition} \label{prop:symetries_dias_gamma}
    For any integer $\gamma \geq 0$, the group of symmetries of
    $\Dias_\gamma$ as a set-operad contains two elements: the identity
    map and the linear map sending any word of $\Dias_\gamma$ to its
    mirror image.
\end{Proposition}
\begin{proof}
    Let us denote by $\Gen_\gamma$ the set $\{0a, a0 : a \in [\gamma]\}$.
    Since $\Dias_\gamma$ is generated by $\Gen_\gamma$, any automorphism
    or antiautomorphism $\phi$ of $\Dias_\gamma$ is wholly determined by
    the images of the elements of $\Gen_\gamma$. Besides let us observe
    that $\phi$ is in particular a permutation of $\Gen_\gamma$.
    \smallskip

    By contradiction, assume that $\phi$ is an automorphism of
    $\Dias_\gamma$ different from the identity map. We have two cases to
    explore.
    \smallskip

    \begin{enumerate}[label={\it Case \arabic*.},fullwidth]
        \item If there are $a, a' \in [\gamma]$ satisfying
        $\phi(0a) = a'0$, since $\phi$ is a permutation of $\Gen_\gamma$,
        there are $b, b' \in [\gamma]$ satisfying $\phi(b0) = 0b'$. Then,
        we have at the same time $b0 \circ_2 0a = b0a = 0a \circ_1 b0$,
        \begin{equation}
            \phi(b0 \circ_2 0a) = \phi(b0) \circ_2 \phi(0a)
            = 0b' \circ_2 a'0 = 0 \, (b'\Max a') \, b',
        \end{equation}
        and
        \begin{equation}
            \phi(0a \circ_1 b0) = \phi(0a) \circ_1 \phi(b0)
            = a'0 \circ_1 0b' = a' \, (a' \Max b') \, 0.
        \end{equation}
        This shows that $\phi(b0 \circ_2 0a) \ne \phi(0a \circ_1 b0)$
        and hence, $\phi$ is not an operad morphism. By a similar
        argument, one can show that there are no $a, a' \in [\gamma]$
        such that $\phi(a0) = 0a'$.
        \smallskip

        \item Otherwise, for all $a \in [\gamma]$, we have $\phi(0a) = 0a'$
        and $\phi(a0) = a''0$ for some $a', a'' \in [\gamma]$. Since, by
        hypothesis, $\phi$ is not the identity map, there exist
        $a \ne a' \in [\gamma]$ such that $\phi(0a) = 0a'$ or
        $\phi(a0) = a'0$. Let us assume, without loss of generality,
        that $\phi(0a) = 0a'$. Since $\phi$ is a permutation of
        $\Gen_\gamma$, there exist $b \ne b' \in [\gamma]$ such that
        $\phi(0b) = 0b'$. One can assume, without loss of generality,
        that $a < b$ and $b' < a'$. Then, we have at the same time
        $0a \circ_2 0b = 0ab = 0b \circ_1 0a$,
        \begin{equation}
            \phi(0a \circ_2 0b) = \phi(0a) \circ_2 \phi(0b)
            = 0a' \circ_2 0b' = 0a'a',
        \end{equation}
        and
        \begin{equation}
            \phi(0b \circ_1 0a) = \phi(0b) \circ_1 \phi(0a)
            = 0b' \circ_1 0a' = 0a'b'.
        \end{equation}
        This shows that $\phi(0a \circ_2 0b) \ne \phi(0b \circ_1 0a)$
        and hence, that $\phi$ is not an operad morphism. By a similar
        argument, one can show that there are no
        $a \ne a' \in [\gamma]$ such that $\phi(a0) = \phi(a'0)$.
    \end{enumerate}
    \smallskip

    We then have shown that if $\phi$ is an automorphism of $\Dias_\gamma$,
    it is necessarily the identity map.
    \smallskip

    Finally, by Proposition~\ref{prop:elements_dias_gamma}, if $x$ is
    an element of $\Dias_\gamma$, its mirror image also is in
    $\Dias_\gamma$. Moreover, it is immediate to see that the map sending
    a word to its mirror image is an antiautomorphism of $\Dias_\gamma$.
    Similar arguments as the ones developed previously show that it is
    the only.
\end{proof}
\medskip

\subsubsection{Basic operad}
A set-operad $\Oca$ is {\em basic} if for all $y_1, \dots, y_n \in \Oca$,
all the maps
\begin{equation}
    \circ^{y_1, \dots, y_n} :
    \Oca(n) \to \Oca(|y_1| + \dots + |y_n|)
\end{equation}
defined by
\begin{equation}
    \circ^{y_1, \dots, y_n}(x) := x \circ (y_1, \dots, y_n),
    \qquad x \in \Oca(n),
\end{equation}
are injective. This property for set-operads introduced by
Vallette~\cite{Val07} is a very relevant one since there is a general
construction producing a family of posets (see~\cite{MY91} and~\cite{CL07})
from a basic set-operad. This family of posets leads to the definition
of an incidence Hopf algebra by a construction of Schmitt~\cite{Sch94}.
\medskip

\begin{Proposition} \label{prop:dias_gamma_basique}
    For any integer $\gamma \geq 0$, $\Dias_\gamma$ is a basic operad.
\end{Proposition}
\begin{proof}
    Let $n \geq 1$, $y_1, \dots, y_n$ be words of $\Dias_\gamma$,
    and $x$ and $x'$ be two words of $\Dias_\gamma(n)$ such that
    $\circ^{y_1, \dots, y_n}(x) = \circ^{y_1, \dots, y_n}(x')$. Then,
    for all $i \in [n]$ and $j \in [|y_i|]$, we have
    $x_i \Max y_{i, j} = x'_i \Max y_{i, j}$ where $y_{i, j}$ is the
    $j$th letter of $y_i$. Since by
    Proposition~\ref{prop:elements_dias_gamma}, any word $y_i$ contains
    a $0$, we have in particular $x_i \Max 0 = x'_i \Max 0$ for all
    $i \in [n]$. This implies $x = x'$ and thus, that
    $\circ^{y_1, \dots, y_n}$ is injective.
\end{proof}
\medskip

\subsubsection{Rooted operad}
We restate here a property on operads introduced by
Chapoton~\cite{Cha14}. An operad $\Oca$ is {\em rooted} if there is a map
\begin{equation}
    \Racine : \Oca(n) \to [n],
    \qquad n \geq 1,
\end{equation}
satisfying, for all $x \in \Oca(n)$, $y \in \Oca(m)$, and $i \in [n]$,
\begin{equation} \label{equ:operade_enracinee}
    \Racine(x \circ_i y) =
    \begin{cases}
        \Racine(x) + m - 1 & \mbox{if } i \leq \Racine(x) - 1, \\
        \Racine(x) + \Racine(y) - 1 & \mbox{if } i = \Racine(x), \\
        \Racine(x) & \mbox{otherwise (} i \geq \Racine(x) + 1 \mbox{)}.
    \end{cases}
\end{equation}
We call such a map a {\em root map}. More intuitively, the root map of a
rooted operad associates a particular input with any of its elements and
this input is preserved by partial compositions.
\medskip

It is immediate that any operad $\Oca$ is a rooted operad for the root
maps $\Racine_{\mathrm{L}}$ and $\Racine_{\mathrm{R}}$, which send
respectively all elements $x$ of arity $n$ to $1$ or to $n$. For this
reason, we say that an operad~$\Oca$ is {\em nontrivially rooted} if
it can be endowed with a root map different from $\Racine_{\mathrm{L}}$
and~$\Racine_{\mathrm{R}}$.
\medskip

\begin{Proposition} \label{prop:dias_gamma_enracinee}
    For any integer $\gamma \geq 0$, $\Dias_\gamma$ is a nontrivially
    rooted operad for the root map sending any word of $\Dias_\gamma$ to
    the position of its $0$.
\end{Proposition}
\begin{proof}
    Thanks to Proposition~\ref{prop:elements_dias_gamma}, the map of the
    statement of the proposition is well-defined. The fact that $0$ is
    the neutral element for the $\Max$ operation and the fact that any
    word of $\Dias_\gamma$ contains exactly one $0$ imply that this map
    satisfies~\eqref{equ:operade_enracinee}. Finally, this map is
    obviously different from $\Racine_{\mathrm{L}}$ and
    $\Racine_{\mathrm{R}}$, whence the statement of the proposition.
\end{proof}
\medskip

\subsubsection{Alternative basis} \label{subsubsec:base_K}
In this section, $\Dias_\gamma$ is considered as an operad in the
category of vector spaces.
\medskip

Let $\OrdDias_\gamma$ be the order relation on the underlying set of
$\Dias_\gamma(n)$, $n \geq 1$, where for all words $x$ and $y$ of
$\Dias_\gamma$ of a same arity $n$, we have
\begin{equation}
    x \OrdDias_\gamma y
    \qquad \mbox{ if } x_i \leq y_i \mbox{ for all } i \in [n].
\end{equation}
This order relation allows to define
for all word $x$ of $\Dias_\gamma$ the elements
\begin{equation} \label{equ:base_K_vers_mots}
    \Ksf^{(\gamma)}_x :=
    \sum_{x \OrdDias_\gamma x'} \mu_\gamma(x, x') \, x',
\end{equation}
where $\mu_\gamma$ is the Möbius function of the poset defined by
$\OrdDias_\gamma$. For instance,
\begin{equation}
    \Ksf^{(2)}_{102} = 102 - 202,
\end{equation}
\begin{equation}
    \Ksf^{(3)}_{102} = \Ksf^{(4)}_{102} = 102 - 103 - 202 + 203,
\end{equation}
\begin{equation}
    \Ksf^{(3)}_{23102}
    = 23102 - 23103 - 23202 + 23203 - 33102 + 33103 + 33202 - 33203.
\end{equation}
\medskip

Since, by Möbius inversion, for any word $x$ of $\Dias_\gamma$ one has
\begin{equation} \label{equ:mots_vers_base_K}
    x = \sum_{x \OrdDias_\gamma x'} \Ksf^{(\gamma)}_{x'},
\end{equation}
the family of all $\Ksf^{(\gamma)}_x$, where the $x$ are words of
$\Dias_\gamma$, forms by triangularity a basis of $\Dias_\gamma$, called
the {\em $\Ksf$-basis}.
\medskip

If $u$ and $v$ are two words of a same length $n$, we denote by
$\Hamming(u, v)$ the {\em Hamming distance} between $u$ and $v$ that is
the number of positions $i \in [n]$ such that $u_i \ne v_i$. Moreover,
for any word $x$ of $\Dias_\gamma$ of length $n$ and any subset $J$ of
$[n]$, we denote by $\Incr_\gamma(x, J)$ the set of words obtained by
incrementing by one some letters of $x$ smaller than $\gamma$ and
greater than $0$ whose positions are in $J$. We shall simply denote by
$\Incr_\gamma(x)$ the set $\Incr_\gamma(x, [n])$.
Proposition~\ref{prop:elements_dias_gamma} ensures that all
$\Incr_\gamma(x, J)$ are sets of words of $\Dias_\gamma$.
\medskip

\begin{Lemme} \label{lem:expression_directe_base_K}
    For any integer $\gamma \geq 0$ and any word $x$ of $\Dias_\gamma$,
    \begin{equation} \label{equ:expression_directe_base_K}
        \Ksf^{(\gamma)}_x =
        \sum_{x' \in \Incr_\gamma(x)} (-1)^{\Hamming(x, x')} \, x'.
    \end{equation}
\end{Lemme}
\begin{proof}
    Let $n$ be the arity of $x$. To compute $\Ksf^{(\gamma)}_x$ from its
    definition~\eqref{equ:base_K_vers_mots}, it is enough to know the
    Möbius function $\mu_\gamma$ of the poset $\Pbb^{(\gamma)}_x$
    consisting in the words $x'$ of $\Dias_\gamma$ satisfying
    $x \OrdDias_\gamma x'$. Immediately from the definition of
    $\OrdDias_\gamma$, it appears that $\Pbb^{(\gamma)}_x$ is isomorphic
    to the Cartesian product poset
    \begin{equation} \label{equ:expression_directe_base_K_poset}
        \Tbb^{(\gamma)}_x :=
        \Tbb\left(\gamma - x_1\right)
        \times \dots \times
        \Tbb\left(\gamma - x_{r - 1}\right)
        \times
        \Tbb(0)
        \times
        \Tbb\left(\gamma - x_{r + 1}\right)
        \times \dots \times
        \Tbb\left(\gamma - x_n\right),
    \end{equation}
    where for any nonnegative integer $k$, $\Tbb(k)$ denotes the poset
    over $\{0\} \cup [k]$ with the natural total order relation, and $r$
    is the position of, by Proposition~\ref{prop:elements_dias_gamma},
    the only $0$ of $x$. The map
    $\phi^{(\gamma)}_x : \Pbb^{(\gamma)}_x \to \Tbb^{(\gamma)}_x$
    defined for all words $x'$ of $\Pbb^{(\gamma)}_x$ by
    \begin{equation}
        \phi^{(\gamma)}_x(x') := \left(x'_1 - x_1, \dots,
        x'_{r - 1} - x_{r - 1}, 0, x'_{r + 1} - x_{r + 1},
        \dots, x'_n - x_n\right)
    \end{equation}
    is an isomorphism of posets.
    \smallskip

    Recall that the Möbius function $\mu$ of $\Tbb(k)$ satisfies, for
    all $a, a' \in \Tbb(k)$,
    \begin{equation}
        \mu(a, a') =
        \begin{cases}
            1 & \mbox{if } a' = a, \\
            -1 & \mbox{if } a' = a + 1, \\
            0 & \mbox{otherwise}.
        \end{cases}
    \end{equation}
    Moreover, since by~\cite{Sta11}, the Möbius function of a Cartesian
    product poset is the product of the Möbius functions of the posets
    involved in the product, through the isomorphism $\phi^{(\gamma)}_x$,
    we obtain that when $x'$ is in $\Incr_\gamma(x)$,
    $\mu_\gamma(x, x') = (-1)^{\Hamming(x, x')}$ and that when $x'$ is
    not in $\Incr_\gamma$, $\mu_\gamma(x, x') = 0$. Therefore,
    \eqref{equ:expression_directe_base_K} is established.
\end{proof}
\medskip

\begin{Lemme} \label{lem:somme_alternee_incr}
    For any integer $\gamma \geq 0$, any word $x$ of $\Dias_\gamma$, and
    any nonempty set $J$ of positions of letters of $x$ that are greater
    than $0$ and smaller than $\gamma$,
    \begin{equation}
        \sum_{x' \in \Incr_\gamma(x, J)} (-1)^{\Hamming(x, x')} = 0.
    \end{equation}
\end{Lemme}
\begin{proof}
    The statement of the lemma follows by induction on the nonzero
    cardinality of $J$.
\end{proof}
\medskip

To compute a direct expression for the partial composition of
$\Dias_\gamma$ over the $\Ksf$-basis, we have to introduce two notations.
If $x$ is a word of $\Dias_\gamma$ of length nonsmaller than $2$, we
denote by $\min(x)$ the smallest letter of $x$ among its letters
different from $0$. Proposition~\ref{prop:elements_dias_gamma} ensures
that $\min(x)$ is well-defined. Moreover, for all words $x$ and $y$ of
$\Dias_\gamma$, a position $i$ such that $x_i \ne 0$, and $a \in [\gamma]$,
we denote by $x \circ_{a, i} y$ the word $x \circ_i y$ in which the $0$
coming from $y$ is replaced by $a$ instead of $x_i$.
\medskip

\begin{Theoreme} \label{thm:composition_base_K}
    For any integer $\gamma \geq 0$, the partial composition of
    $\Dias_\gamma$ over the $\Ksf$-basis satisfies, for all words $x$
    and $y$ of $\Dias_\gamma$ of arities nonsmaller than $2$,
    \begin{equation}
        \Ksf^{(\gamma)}_x \circ_i \Ksf^{(\gamma)}_y =
        \begin{cases}
            \Ksf^{(\gamma)}_{x \circ_i y} & \mbox{if } \min(y) > x_i, \\
            \sum_{a \in [x_i, \gamma]} \Ksf^{(\gamma)}_{x \circ_{a, i} y}
                & \mbox{if } \min(y) = x_i, \\
            0 & \mbox{otherwise (} \min(y) < x_i \mbox{)}.
        \end{cases}
    \end{equation}
\end{Theoreme}
\begin{proof}
    First of all, by Lemma~\ref{lem:expression_directe_base_K} together
    with~\eqref{equ:mots_vers_base_K}, we obtain
    \begin{equation} \label{equ:demo_composition_base_K} \begin{split}
        \Ksf^{(\gamma)}_x \circ_i \Ksf^{(\gamma)}_y
            & = \sum_{\substack{x' \in \Incr_\gamma(x) \\
                                y' \in \Incr_\gamma(y)}}
            (-1)^{\Hamming(x, x') + \Hamming(y, y')}
            \left(\sum_{x' \circ_i y' \OrdDias_\gamma z}
            \Ksf^{(\gamma)}_z\right) \\
            & = \sum_{x \circ_i y \OrdDias_\gamma z} \;
            \sum_{\substack{x' \in \Incr_\gamma(x) \\
                                  y' \in \Incr_\gamma(y) \\
                                  x' \circ_i y' \OrdDias_\gamma z}}
            (-1)^{\Hamming(x, x') + \Hamming(y, y')}
            \, \Ksf^{(\gamma)}_z.
    \end{split} \end{equation}
    \smallskip

    Let us denote by $n$ (resp. $m$) the arity of $x$ (resp. $y$) and
    let $z$ be a word of $\Dias_\gamma$ such that
    $x \circ_i y \OrdDias_\gamma z$. Let $x' \in \Incr_\gamma(x)$ and
    $y' \in \Incr_\gamma(y)$. We have, by definition of the partial
    composition of $\Dias_\gamma$,
    \begin{equation}
        x \circ_i y =
        x_1 \dots x_{i - 1}
            \, t_1 \dots t_{r - 1} \, x_i \, t_{r + 1} \dots t_m \,
        x_{i + 1} \dots x_n,
    \end{equation}
    and
    \begin{equation}
        x' \circ_i y' =
        x'_1 \dots x'_{i - 1}
            \, t'_1 \dots t'_{r - 1} \, x'_i \, t'_{r + 1} \dots t'_m \,
        x'_{i + 1} \dots x'_n,
    \end{equation}
    where $r$ denotes the position of the only, by
    Proposition~\ref{prop:elements_dias_gamma}, $0$ of $y$ and for all
    $j \in [m] \setminus \{r\}$, $t_j := x_i \Max y_j$ and
    $t'_j := x'_i \Max y'_j$. By~\eqref{equ:demo_composition_base_K},
    the pair $(x', y')$ contributes to the coefficient of $\Ksf^{(\gamma)}_z$
    in~\eqref{equ:demo_composition_base_K} if and only if
    $x \circ_i y \OrdDias_\gamma x' \circ_i y' \OrdDias z$. To compute
    this coefficient, we have three cases to consider following the value
    of $\min(y)$ compared to the value of~$x_i$.
    \smallskip

    \begin{enumerate}[label={\it Case \arabic*.},fullwidth]
        \item Assume first that $\min(y) < x_i$. Then, there is at least
        a $s \in [m] \setminus \{r\}$ such that $y_s < x_i$. This implies
        that $t_s = x_i$ and that $y'_s$ has no influence on $t'_s$ and
        then, on $x' \circ_i y'$. Thus, the word
        $y'' := y'_1 \dots y'_{s - 1} a y'_{s + 1} \dots y'_m$ where $a$
        is the only possible letter such that $y'' \in \Incr_\gamma(y)$
        and $a \ne y'_s$ satisfies $x' \circ_i y'' = x' \circ_i y'$.
        Therefore, since $\Hamming(y', y'') = 1$, the contribution of
        the pair $(x', y')$ for the coefficient of $\Ksf^{(\gamma)}_z$
        in~\eqref{equ:demo_composition_base_K} is compensated by the
        contribution of the pair $(x', y'')$. This shows that this
        coefficient is $0$ and hence,
        $\Ksf^{(\gamma)}_x \circ_i \Ksf^{(\gamma)}_y = 0$.
        \smallskip

        \item \label{item:composition_base_K_cas_2}
        Assume now that $\min(y) > x_i$. Then, for all
        $j \in [m] \setminus \{r\}$, we have $y_j > x_i$ and thus,
        $t_j = y_j$. When $z = x \circ_i y$, we necessarily have $x' = x$
        and $y' = y$. Hence, the coefficient of $\Ksf^{(\gamma)}_{x \circ_i y}$
        in~\eqref{equ:demo_composition_base_K} is $1$. Else, when
        $z \ne x \circ_i y$, we have
        $x' \circ_i y' \in \Incr_\gamma(x \circ_i y, J)$, where $J$ is
        the nonempty set of the positions of letters of $z$ different
        from letters of $x \circ_i y$. Now,
        from~\eqref{equ:demo_composition_base_K}, the coefficient of
        $\Ksf^{(\gamma)}_z$ in~\eqref{equ:demo_composition_base_K} is
        \begin{equation}
            \sum_{x' \circ_i y' \in \Incr_\gamma(x \circ_i y, J)}
            (-1)^{\Hamming(x, x') + \Hamming(y, y')}.
        \end{equation}
        Lemma~\ref{lem:somme_alternee_incr} implies that this coefficient
        is $0$. This shows that
        $\Ksf^{(\gamma)}_x \circ_i \Ksf^{(\gamma)}_y =
        \Ksf^{(\gamma)}_{x \circ_i y}$.
        \smallskip

        \item The last case occurs when $\min(y) = x_i$. Then, for all
        $j \in [m] \setminus \{r\}$, we have $y_j \geq x_i$ and thus,
        $t_j = y_j$. Moreover, there is at least a $s \in [m] \setminus \{r\}$
        such that $y_s = x_i$. When $z = x \circ_{a, i} y$ with
        $a \in [x_i, \gamma]$, we necessarily have $x' = x$ and $y' = y$.
        Therefore, for all $a \in [x_i, \gamma]$, the
        $\Ksf^{(\gamma)}_{x \circ_{a, i}}$ have coefficient~$1$
        in~\eqref{equ:demo_composition_base_K}. The same argument as the
        one exposed for \ref{item:composition_base_K_cas_2} shows that
        when $z \ne x \circ_{a, i} y$ for all $a \in [x_i, \gamma]$, the
        coefficient of $\Ksf^{(\gamma)}_z$ is zero. Hence,
        $\Ksf^{(\gamma)}_x \circ_i \Ksf^{(\gamma)}_y =
        \sum_{a \in [x_i, \gamma]} \Ksf^{(\gamma)}_{x \circ_{a, i} y}$.
    \end{enumerate}
\end{proof}
\medskip

We have for instance
\begin{equation}
    \Ksf^{(5)}_{20413} \circ_1 \Ksf^{(5)}_{304} = \Ksf^{(5)}_{3240413},
\end{equation}
\begin{equation}
    \Ksf^{(5)}_{20413} \circ_2 \Ksf^{(5)}_{304} = \Ksf^{(5)}_{2304413},
\end{equation}
\begin{equation}
    \Ksf^{(5)}_{20413} \circ_3 \Ksf^{(5)}_{304} = 0,
\end{equation}
\begin{equation}
    \Ksf^{(5)}_{20413} \circ_4 \Ksf^{(5)}_{304} = \Ksf^{(5)}_{2043143},
\end{equation}
\begin{equation}
    \Ksf^{(5)}_{20413} \circ_5 \Ksf^{(5)}_{304}
        = \Ksf^{(5)}_{2041334} + \Ksf^{(5)}_{2041344}
        + \Ksf^{(5)}_{2041354}.
\end{equation}
\medskip

Theorem~\ref{thm:composition_base_K} implies in particular that the
structure coefficients of the partial composition of $\Dias_\gamma$ over
the $\Ksf$-basis are $0$ or $1$. It is possible to define another bases
of $\Dias_\gamma$ by reversing in~\eqref{equ:base_K_vers_mots} the
relation $\OrdDias_\gamma$ and by suppressing or keeping the Möbius
function $\mu_\gamma$. This gives obviously rise to three other bases.
It worth to note that, as small computations reveal, over all these
additional bases, the structure coefficients of the partial composition
of $\Dias_\gamma$ can be negative or different from $1$. This observation
makes the $\Ksf$-basis even more particular and interesting. It has some
other properties, as next section will show.
\medskip

\subsubsection{Alternative presentation}%
\label{subsubsec:presentation_dias_gamma_alternative}
The $\Ksf$-basis introduced in the previous section leads to state a new
presentation for $\Dias_\gamma$ in the following way.
\medskip

For any integer $\gamma \geq 0$, let $\GDiasA_a$ and $\DDiasA_a$,
$a \in [\gamma]$, be the elements of $\OpLibre\left(\GenDias\right)(2)$
defined by
\begin{subequations}
\begin{equation}
    \GDiasA_a :=
    \begin{cases}
        \GDias_\gamma & \mbox{if } a = \gamma, \\
        \GDias_a - \GDias_{a + 1} & \mbox{otherwise},
    \end{cases}
\end{equation}
and
\begin{equation}
    \DDiasA_a :=
    \begin{cases}
        \DDias_\gamma & \mbox{if } a = \gamma, \\
        \DDias_a - \DDias_{a + 1} & \mbox{otherwise}.
    \end{cases}
\end{equation}
\end{subequations}
Then, since for all $a \in [\gamma]$ we have
\begin{subequations}
\begin{equation}
    \GDias_a = \sum_{a \leq b \in [\gamma]} \GDiasA_b
\end{equation}
and
\begin{equation}
    \DDias_a = \sum_{a \leq b \in [\gamma]} \DDiasA_b,
\end{equation}
\end{subequations}
by triangularity, the family
$\GenDias' := \{\GDiasA_a, \DDiasA_a : a \in [\gamma]\}$
forms a basis of $\OpLibre\left(\GenDias\right)(2)$ and then, generates
$\OpLibre\left(\GenDias\right)$ as an operad. This change of basis
from $\OpLibre\left(\GenDias\right)$ to $\OpLibre(\GenDias')$
comes from the change of basis from the usual basis of $\Dias_\gamma$
to the $\Ksf$-basis. Let us now express a presentation of $\Dias_\gamma$
through the family $\GenDias'$.
\medskip

\begin{Proposition} \label{prop:presentation_alternative_dias_gamma}
    For any integer $\gamma \geq 0$, the operad $\Dias_\gamma$ admits
    the following presentation. It is generated by $\GenDias'$ and its
    space of relations is $\RelLibre'_{\Dias_\gamma}$ is generated by
    \begin{subequations}
    \begin{equation} \label{equ:relation_dias_gamma_1_alternative}
        \GDiasA_a \circ_1 \DDiasA_{a'} -
        \DDiasA_{a'} \circ_2 \GDiasA_a,
        \qquad a, a' \in [\gamma],
    \end{equation}
    \begin{equation} \label{equ:relation_dias_gamma_2_alternative}
        \DDiasA_b \circ_1 \DDiasA_a, \qquad a < b \in [\gamma],
    \end{equation}
    \begin{equation} \label{equ:relation_dias_gamma_3_alternative}
        \GDiasA_b \circ_2 \GDiasA_a, \qquad a < b \in [\gamma],
    \end{equation}
    \begin{equation} \label{equ:relation_dias_gamma_4_alternative}
        \DDiasA_b \circ_1 \GDiasA_a, \qquad a < b \in [\gamma],
    \end{equation}
    \begin{equation} \label{equ:relation_dias_gamma_5_alternative}
        \GDiasA_b \circ_2 \DDiasA_a, \qquad a < b \in [\gamma],
    \end{equation}
    \begin{equation} \label{equ:relation_dias_gamma_6_alternative}
        \DDiasA_a \circ_1 \DDiasA_b
        - \DDiasA_b \circ_2 \DDiasA_a,
        \qquad a < b \in [\gamma],
    \end{equation}
    \begin{equation} \label{equ:relation_dias_gamma_7_alternative}
        \GDiasA_b \circ_1 \GDiasA_a
        - \GDiasA_a \circ_2 \GDiasA_b,
        \qquad a < b \in [\gamma],
    \end{equation}
    \begin{equation} \label{equ:relation_dias_gamma_8_alternative}
        \DDiasA_a \circ_1 \GDiasA_b
        - \DDiasA_a \circ_2 \DDiasA_b,
        \qquad a < b \in [\gamma],
    \end{equation}
    \begin{equation} \label{equ:relation_dias_gamma_9_alternative}
        \GDiasA_a \circ_1 \GDiasA_b
        - \GDiasA_a \circ_2 \DDiasA_b,
        \qquad a < b \in [\gamma],
    \end{equation}
    \begin{equation} \label{equ:relation_dias_gamma_10_alternative}
        \DDiasA_a \circ_1 \DDiasA_a -
        \left(\sum_{a \leq b \in [\gamma]} \DDiasA_a \circ_2 \DDiasA_b\right),
        \qquad a \in [\gamma],
    \end{equation}
    \begin{equation} \label{equ:relation_dias_gamma_11_alternative}
        \left(\sum_{a \leq b \in [\gamma]} \GDiasA_a \circ_1 \GDiasA_b\right)
        - \GDiasA_a \circ_2 \GDiasA_a,
        \qquad a \in [\gamma],
    \end{equation}
    \begin{equation} \label{equ:relation_dias_gamma_12_alternative}
        \DDiasA_a \circ_1 \GDiasA_a
        - \left(\sum_{a \leq b \in [\gamma]} \DDiasA_b \circ_2 \DDiasA_a\right),
        \qquad a \in [\gamma],
    \end{equation}
    \begin{equation} \label{equ:relation_dias_gamma_13_alternative}
        \left(\sum_{a \leq b \in [\gamma]} \GDiasA_b \circ_1 \GDiasA_a\right)
        - \GDiasA_a \circ_2 \DDiasA_a,
        \qquad a \in [\gamma].
    \end{equation}
    \end{subequations}
\end{Proposition}
\begin{proof}
    Let us show that $\RelLibre'_{\Dias_\gamma}$ is equal to the space
    of relations $\RelLibre_{\Dias_\gamma}$ of $\Dias_\gamma$ defined in
    the statement of Theorem~\ref{thm:presentation_dias_gamma}. First of
    all, recall that the map
    $\Mot_\gamma : \OpLibre\left(\GenDias\right) \to \Dias_\gamma$
    defined in Section~\ref{subsubsec:arbres_vers_mots} satisfies
    $\Mot_\gamma(\GDias_a) = 0a$ and $\Mot_\gamma(\DDias_a) = a0$ for
    all $a \in [\gamma]$. By Theorem~\ref{thm:presentation_dias_gamma},
    for any $x \in  \OpLibre\left(\GenDias\right)(3)$, $x$ is in
    $\RelLibre_{\Dias_\gamma}$ if and only if $\Mot_\gamma(x) = 0$.
    \smallskip

    Besides, by definition of $\GDiasA_a$, $\DDiasA_a$, $a \in [\gamma]$,
    and by making use of the $\Ksf$-basis of $\Dias_\gamma$, we have
    $\Mot_\gamma(\GDiasA_a) = \Ksf^{(\gamma)}_{0a}$ and
    $\Mot_\gamma(\DDiasA_a) = \Ksf^{(\gamma)}_{a0}$. By using the partial
    composition rules for $\Dias_\gamma$ over the $\Ksf$-basis of
    Theorem~\ref{thm:composition_base_K}, straightforward computations
    show that $\Mot_\gamma(x) = 0$ for all elements $x$
    among~\eqref{equ:relation_dias_gamma_1_alternative}---%
    \eqref{equ:relation_dias_gamma_13_alternative}.
    This implies that $\RelLibre'_{\Dias_\gamma}$ is a subspace
    of $\RelLibre_{\Dias_\gamma}$.
    \smallskip

    Now, one can observe that
    elements~\eqref{equ:relation_dias_gamma_1_alternative}---%
    \eqref{equ:relation_dias_gamma_13_alternative} are
    linearly independent. Then, $\RelLibre'_{\Dias_\gamma}$ has
    dimension $5 \gamma^2$ which is also, by
    Theorem~\ref{thm:presentation_dias_gamma}, the dimension of
    $\RelLibre_{\Dias_\gamma}$. Hence, $\RelLibre'_{\Dias_\gamma}$ and
    $\RelLibre_{\Dias_\gamma}$ are equal. The statement of the
    proposition follows.
\end{proof}
\medskip

Despite the apparent complexity of the presentation of $\Dias_\gamma$
exhibited by Proposition~\ref{prop:presentation_alternative_dias_gamma},
as we will see in Section~2 of~\cite{GirII}, the
Koszul dual of $\Dias_\gamma$ computed from this presentation has a very
simple and manageable expression.
\medskip

\section{Pluriassociative algebras} \label{sec:algebres_dias_gamma}
We now focus on algebras over $\gamma$-pluriassociative operads.
For this purpose, we construct free $\Dias_\gamma$-algebras over one
generator, and define and study two notions of units for
$\Dias_\gamma$-algebras. We end this section by introducing a convenient
way to define $\Dias_\gamma$-algebras and give several examples of such
algebras.
\medskip

\subsection{Category of pluriassociative algebras and free objects}
Let us study the category of $\Dias_\gamma$-algebras and the units for
algebras in this category.
\medskip

\subsubsection{Pluriassociative algebras}
We call {\em $\gamma$-pluriassociative algebra} any
$\Dias_\gamma$-algebra. From the presentation of $\Dias_\gamma$ provided
by Theorem~\ref{thm:presentation_dias_gamma}, any
$\gamma$-pluriassociative algebra is a vector space endowed with linear
operations $\GDias_a, \DDias_a$, $a \in [\gamma]$, satisfying
the relations encoded by~\eqref{equ:relation_dias_gamma_1_concise}---%
\eqref{equ:relation_dias_gamma_5_concise}.
\medskip

\subsubsection{General definitions}
Let $\Pca$ be a $\gamma$-pluriassociative algebra. We say that $\Pca$
is {\em commutative} if for all $x, y \in \Pca$ and $a \in [\gamma]$,
$x \GDias_a y = y \DDias_a x$. Besides, $\Pca$ is {\em pure} for all
$a, a' \in [\gamma]$, $a \ne a'$ implies $\GDias_a \ne \GDias_{a'}$ and
$\DDias_a \ne \DDias_{a'}$.
\medskip

Given a subset $C$ of $[\gamma]$, one can keep on the vector space $\Pca$
only the operations $\GDias_a$ and $\DDias_a$ such that $a \in C$.
By renumbering the indexes of these operations from $1$ to $\# C$ by
respecting their former relative numbering, we obtain a
$\# C$-pluriassociative algebra. We call it the
{\em $\# C$-pluriassociative subalgebra induced by} $C$ of $\Pca$.
\medskip

\subsubsection{Free pluriassociative algebras}
\label{subsubsec:algebre_dias_gamma_libre}
Recall that $\AlgLibre_{\Dias_\gamma}$ denotes the free
$\Dias_\gamma$-algebra over one generator. By definition,
$\AlgLibre_{\Dias_\gamma}$ is the linear span of the set of the words on
$\{0\} \cup [\gamma]$ with exactly one occurrence of~$0$. Let us endow
this space with the linear operations
\begin{equation}
    \GDias_a, \DDias_a :
    \AlgLibre_{\Dias_\gamma} \otimes \AlgLibre_{\Dias_\gamma}
    \to \AlgLibre_{\Dias_\gamma},
    \qquad
    a \in [\gamma],
\end{equation}
satisfying, for any such words $u$ and $v$,
\begin{subequations}
\begin{equation}
    u \GDias_a v := u \: \Augm_a(v)
\end{equation}
and
\begin{equation}
    u \DDias_a v := \Augm_a(u) \: v,
\end{equation}
\end{subequations}
where $\Augm_a(u)$ (resp. $\Augm_a(v)$) is the word obtained by replacing
in $u$ (resp. $v$) any occurrence of a letter smaller than $a$ by $a$.
\medskip

\begin{Proposition} \label{prop:algebre_dias_gamma_libre}
    For any integer $\gamma \geq 0$, the vector space
    $\AlgLibre_{\Dias_\gamma}$ of nonempty words on $\{0\} \cup [\gamma]$
    containing exactly one occurrence of $0$ endowed with the operations
    $\GDias_a$, $\DDias_a$, $a \in [\gamma]$, is the free
    $\gamma$-pluriassociative algebra over one generator.
\end{Proposition}
\begin{proof}
    The fact that $\AlgLibre_{\Dias_\gamma}$ is the stated vector space
    is a consequence of the description of the elements of $\Dias_\gamma$
    provided by Proposition~\ref{prop:elements_dias_gamma}. Since
    $\Dias_\gamma$ is by definition the suboperad of $\T \Mca_\gamma$
    generated by $\{0a, a0 : a \in [\gamma]\}$, $\AlgLibre_{\Dias_\gamma}$
    is endowed with $2\gamma$ binary operations where any generator
    $0a$ (resp. $a0$) gives rise to the operation $\GDias_a$ (resp.
    $\DDias_a$) of $\AlgLibre_{\Dias_\gamma}$. Moreover, by making use
    of the realization of $\Dias_\gamma$, we have for all
    $u, v \in \AlgLibre_{\Dias_\gamma}$ and $a \in [\gamma]$,
    \begin{subequations}
    \begin{equation}
        u \GDias_a v
        = (u \otimes v) \Action 0a
        = (0a \circ_2 v) \circ_1 u = u \: \Augm_a(v)
    \end{equation}
    and
    \begin{equation}
        u \DDias_a v
        = (u \otimes v) \Action a0
        = (a0 \circ_2 v) \circ_1 u = \Augm_a(u) \: v.
    \end{equation}
    \end{subequations}
\end{proof}
\medskip

One has for instance in $\AlgLibre_{\Dias_4}$,
\begin{equation}
    \textcolor{Bleu}{101241} \GDias_2 \textcolor{Rouge}{203} =
    \textcolor{Bleu}{101241}\textcolor{Rouge}{2{\bf 2}3}
\end{equation}
and
\begin{equation}
    \textcolor{Bleu}{101241} \DDias_3 \textcolor{Rouge}{203} =
    \textcolor{Bleu}{{\bf 3333}4{\bf 3}}\textcolor{Rouge}{203}.
\end{equation}
\medskip

\subsection{Bar and wire-units}
Loday has defined in~\cite{Lod01} some notions of units in diassociative
algebras. We generalize here these definitions to the context of
$\gamma$-pluriassociative algebras.
\medskip

\subsubsection{Bar-units}
Let $\Pca$ be a $\gamma$-pluriassociative algebra and $a \in [\gamma]$.
We say that an element $e$ of $\Pca$ is an {\em $a$-bar-unit}, or
simply a {\em bar-unit} when taking into account the value of $a$
is not necessary, of $\Pca$ if for all $x \in \Pca$,
\begin{equation}
    x \GDias_a e = x = e \DDias_a x.
\end{equation}
As we shall see below, a $\gamma$-pluriassociative algebra can have,
for a given $a \in [\gamma]$, several $a$-bar-units. The {\em $a$-halo}
of $\Pca$, denoted by $\Halo_a(\Pca)$, is the set of the $a$-bar-units
of $\Pca$.
\medskip

\subsubsection{Wire-units}
Let $\Pca$ be a $\gamma$-pluriassociative algebra and $a \in [\gamma]$.
We say that an element $e$ of $\Pca$ is an {\em $a$-wire-unit}, or
simply a {\em wire-unit} when taking into account the value of $a$
is not necessary, of $\Pca$ if for all $x \in \Pca$,
\begin{equation}
    e \GDias_a x = x = x \DDias_a e.
\end{equation}
As shows the following proposition, the presence of a wire-unit in
$\Pca$ has some implications.
\medskip

\begin{Proposition} \label{prop:unite_fil}
    Let $\gamma \geq 0$ be an integer and $\Pca$ be a
    $\gamma$-pluriassociative algebra admitting a $b$-wire-unit $e$
    for a $b \in [\gamma]$. Then
    \begin{enumerate}[label={\it (\roman*)}]
        \item \label{prop:unite_fil_1}
        for all $a \in [b]$, the operations $\GDias_a$,
        $\GDias_b$, $\DDias_a$, and $\DDias_b$ of $\Pca$ are equal;
        \item \label{prop:unite_fil_2}
        $e$ is also an $a$-wire-unit for all $a \in [b]$;
        \item \label{prop:unite_fil_3}
        $e$ is the only wire-unit of $\Pca$;
        \item \label{prop:unite_fil_4}
        if $e'$ is an $a$-bar unit for a $a \in [b]$, then $e' = e$.
    \end{enumerate}
\end{Proposition}
\begin{proof}
    Let us show part~\ref{prop:unite_fil_1}. By
    Relation~\eqref{equ:relation_dias_gamma_4_concise} of
    $\gamma$-pluriassociative algebras and by the fact that $e$ is a
    $b$-wire-unit of $\Pca$, we have for all elements $y$ and $z$
    of $\Pca$ and all $a \in [b]$,
    \begin{equation}
        y \GDias_a z = e \GDias_b (y \GDias_a z) =
        e \GDias_b (y \DDias_a z) = y \DDias_a z.
    \end{equation}
    Thus, the operations $\GDias_a$ and $\DDias_a$ of $\Pca$ are equal.
    Moreover, for the same reasons, we have
    \begin{equation}
        y \GDias_a z = e \GDias_b (y \GDias_a z) =
        (e \GDias_b y) \GDias_b z = y \GDias_b z.
    \end{equation}
    Then, the operations $\GDias_a$ and $\GDias_b$ of $\Pca$ are equal,
    whence~\ref{prop:unite_fil_1}.
    \smallskip

    Now, by~\ref{prop:unite_fil_1} and by the fact that $e$ is a
    $b$-wire-unit, we have for all elements $x$ of $\Pca$ and all
    $a \in [b]$,
    \begin{equation}
        e \GDias_a x = e \GDias_b x = x = x \DDias_b e = x \DDias_a e,
    \end{equation}
    showing~\ref{prop:unite_fil_2}.
    \smallskip

    To prove~\ref{prop:unite_fil_3}, assume that $e'$ is a $b'$-wire-unit
    of $\Pca$ for a $b' \in [\gamma]$. By~\ref{prop:unite_fil_1} and by
    the fact that $e$ is a $b$-wire-unit, one has
    \begin{equation}
        e = e \DDias_{b'} e' = e \GDias_b e' = e',
    \end{equation}
    showing~\ref{prop:unite_fil_3}.
    \smallskip

    To establish~\ref{prop:unite_fil_4}, let us first prove that $e$ is
    a $b$-bar-unit. By~\ref{prop:unite_fil_1} and by the fact that $e$
    is a $b$-wire-unit, we have for all elements $x$ of $\Pca$,
    \begin{equation}
        e \DDias_b x = e \GDias_b x = x = x \DDias_b e = x \GDias_b e.
    \end{equation}
    Now, since $e'$ is an $a$-bar-unit for an $a \in [b]$,
    by~\ref{prop:unite_fil_1} and by the fact that $e$ is a $b$-wire-unit,
    \begin{equation}
        e = e' \DDias_a e = e' \DDias_b e = e'.
    \end{equation}
    This shows~\ref{prop:unite_fil_4}.
\end{proof}
\medskip

Relying on Proposition~\ref{prop:unite_fil}, we define the {\em height}
of a $\gamma$-pluriassociative algebra $\Pca$ as zero if $\Pca$ has no
wire-unit, otherwise as the greatest integer $h \in [\gamma]$ such that
the unique wire-unit $e$ of $\Pca$ is a $h$-wire-unit. Observe that any
pure $\gamma$-pluriassociative algebra has height $0$ or $1$.
\medskip

\subsection{Construction of pluriassociative algebras}
We now present a general way to construct $\gamma$-pluriassociative
algebras. Our construction is a natural generalization of some
constructions introduced by Loday~\cite{Lod01} in the context of
diassociative algebras. We introduce in this section new algebraic
structures, the so-called $\gamma$-multiprojection algebras, which are
inputs of our construction.
\medskip

\subsubsection{Multiassociative algebras}%
\label{subsubsec:algebres_multiassociatives}
For any integer $\gamma \geq 0$, a {\em $\gamma$-multiassociative algebra}
is a vector space $\Mca$ endowed with linear operations
\begin{equation}
    \MAs_a : \Mca \otimes \Mca \to \Mca,
    \qquad a \in [\gamma],
\end{equation}
satisfying, for all $x, y, z \in \Mca$, the relations
\begin{equation} \label{equ:relation_algebre_multiassoc}
    (x \MAs_a y) \MAs_b z =
    (x \MAs_b y) \MAs_{a'} z =
    x \MAs_{a''} (y \MAs_b z) =
    x \MAs_b (y \MAs_{a'''} z),
    \qquad a, a', a'', a''' \leq b \in [\gamma].
\end{equation}
These algebras are obvious generalizations of associative algebras since
all of its operations are associative. Observe that
by~\eqref{equ:relation_algebre_multiassoc}, all bracketings of
an expression involving elements of a $\gamma$-multiassociative algebra
and some of its operations are equal. Then, since the bracketings of such
expressions are not significant, we shall denote these without parenthesis.
In Section~3 of~\cite{GirII}, we will study the
underlying operads of the category of $\gamma$-multiassociative algebras,
called $\As_\gamma$, for a very specific purpose.
\medskip

If $\Mca_1$ and $\Mca_2$ are two $\gamma$-multiassociative algebras,
a linear map $\phi : \Mca_1 \to \Mca_2$ is a
{\em $\gamma$-multiassociative algebra morphism} if it commutes with
the operations of $\Mca_1$ and $\Mca_2$. We say that $\Mca$ is
{\em commutative} when all operations of $\Mca$ are commutative.
Besides, for an $a \in [\gamma]$, an element $\Unite$ of $\Mca$ is an
{\em $a$-unit}, or simply a {\em unit} when taking into account the
value of $a$ is not necessary, of $\Mca$ if for all $x \in \Mca$,
$\Unite \MAs_a x = x = x \MAs_a \Unite$. When $\Mca$ admits a unit,
we say that $\Mca$ is {\em unital}. As shows the following proposition,
the presence of a unit in $\Mca$ has some implications.
\medskip

\begin{Proposition} \label{prop:unites_algebre_multiassociative}
    Let $\gamma \geq 0$ be an integer and $\Mca$ be a
    $\gamma$-multiassociative algebra admitting a $b$-unit $\Unite$
    for a $b \in [\gamma]$. Then
    \begin{enumerate}[label={\it (\roman*)}]
        \item \label{prop:unites_algebre_multiassociative_1}
        for all $a \in [b]$, the operations $\MAs_a$ and
        $\MAs_b$ of $\Mca$ are equal;
        \item \label{prop:unites_algebre_multiassociative_2}
        $\Unite$ is also an $a$-unit for all $a \in [b]$;
        \item \label{prop:unites_algebre_multiassociative_3}
        $\Unite$ is the only unit of $\Mca$.
    \end{enumerate}
\end{Proposition}
\begin{proof}
    By Relation~\eqref{equ:relation_algebre_multiassoc} of
    $\gamma$-multiassociative algebras and by the fact that $\Unite$ is
    a $b$-unit of $\Mca$, we have for all elements $y$ and $z$ of $\Mca$
    and all $a \in [b]$,
    \begin{equation}
        y \MAs_a z
            = y \MAs_a z \MAs_b \Unite
            = y \MAs_b z \MAs_b \Unite
            = y \MAs_b z.
    \end{equation}
    Therefore, $\MAs_a = \MAs_b$,
    showing~\ref{prop:unites_algebre_multiassociative_1}.
    \smallskip

    Now, by~\ref{prop:unites_algebre_multiassociative_1} and by the fact
    that $\Unite$ is a $b$-unit, we have for all elements $x$ of $\Mca$
    and all $a \in [b]$,
    \begin{equation}
        \Unite \MAs_a x
            = \Unite \MAs_b x
            = x
            = x \MAs_b \Unite = x \MAs_a \Unite,
    \end{equation}
    showing~\ref{prop:unites_algebre_multiassociative_2}.
    \smallskip

    To prove~\ref{prop:unites_algebre_multiassociative_3}, assume
    that $\Unite'$ is a $b'$-unit of $\Mca$ for a $b' \in [\gamma]$.
    By~\ref{prop:unites_algebre_multiassociative_1} and by the fact that
    $\Unite$ is a $b$-unit, one has
    \begin{equation}
        \Unite
            = \Unite \MAs_{b'} \Unite'
            = \Unite \MAs_b \Unite'
            = \Unite',
    \end{equation}
    establishing~\ref{prop:unites_algebre_multiassociative_3}.
\end{proof}
\medskip

Relying on Proposition~\ref{prop:unites_algebre_multiassociative},
similarly to the case of $\gamma$-pluriassociative algebras, we define
the {\em height} of a $\gamma$-multiassociative algebra $\Mca$ as zero
if $\Mca$ has no unit, otherwise as the greatest integer $h \in [\gamma]$
such that the unit $\Unite$ of $\Mca$ is an $h$-unit.
\medskip

\subsubsection{Multiprojection algebras}
We call {\em $\gamma$-multiprojection algebra} any $\gamma$-multiassociative
algebra $\Mca$ endowed with endomorphisms
\begin{equation}
    \pi_a : \Mca \to \Mca,
    \qquad a \in [\gamma],
\end{equation}
satisfying
\begin{equation} \label{equ:relation_algebre_multiproj}
    \pi_a \circ \pi_{a'} = \pi_{a \Max a'},
    \qquad a, a' \in [\gamma].
\end{equation}
\medskip

By extension, the {\em height} of $\Mca$ is its height as a
$\gamma$-multiassociative algebra. We say that $\Mca$ is {\em unital} as
a $\gamma$-multiprojection algebra if $\Mca$ is unital as a
$\gamma$-multiassociative algebra and its only, by
Proposition~\ref{prop:unites_algebre_multiassociative}, unit $\Unite$
satisfies $\pi_a(\Unite) = \Unite$ for all $a \in [h]$ where $h$ is the
height of~$\Mca$.
\medskip

\subsubsection{From multiprojection algebras to pluriassociative algebras}
Next result describes how to construct $\gamma$-pluriassociative
algebras from $\gamma$-multiprojection algebras.
\medskip

\begin{Theoreme} \label{thm:algebre_multiprojections_vers_dias_gamma}
    For any integer $\gamma \geq 0$ and any $\gamma$-multiprojection
    algebra $\Mca$, the vector space $\Mca$ endowed with binary linear
    operations $\GDias_a$, $\DDias_a$, $a \in [\gamma]$, defined for all
    $x, y \in \Mca$ by
    \begin{subequations}
    \begin{equation}
        x \GDias_a y := x \MAs_a \pi_a(y)
    \end{equation}
    and
    \begin{equation}
        x \DDias_a y := \pi_a(x) \MAs_a y,
    \end{equation}
    \end{subequations}
    where the $\MAs_a$, $a \in [\gamma]$, are the operations of $\Mca$ and
    the $\pi_a$, $a \in [\gamma]$, are its endomorphisms, is a
    $\gamma$-pluriassociative algebra, denoted by $\MProjVersPluri(\Mca)$.
\end{Theoreme}
\begin{proof}
    This is a verification of the relations of $\gamma$-pluriassociative
    algebras in $\MProjVersPluri(\Mca)$. Let $x$, $y$, and $z$ be three
    elements of $\MProjVersPluri(\Mca)$ and  $a, a' \in [\gamma]$.
    \smallskip

    By~\eqref{equ:relation_algebre_multiassoc}, we have
    \begin{equation}
        (x \DDias_{a'} y) \GDias_a z =
        \pi_{a'}(x) \MAs_{a'} y \MAs_a \pi_a(z) =
        x \DDias_{a'} (y \GDias_a z),
    \end{equation}
    showing that~\eqref{equ:relation_dias_gamma_1_concise} is
    satisfied in $\MProjVersPluri(\Mca)$.
    \smallskip

    Moreover, by~\eqref{equ:relation_algebre_multiassoc}
    and~\eqref{equ:relation_algebre_multiproj}, we have
    \begin{equation} \begin{split}
        x \GDias_a (y \DDias_{a'} z)
            & = x \MAs_a \pi_a(\pi_{a'}(y) \MAs_{a'} z) \\
            & = x \MAs_a \pi_{a \Max a'}(y) \MAs_{a'} \pi_a(z) \\
            & = x \MAs_{a \Max a'} \pi_{a \Max a'}(y) \MAs_a \pi_a(z) \\
            & = (x \GDias_{a \Max a'} y) \GDias_a z,
    \end{split} \end{equation}
    so that~\eqref{equ:relation_dias_gamma_2_concise},
    and for the same
    reasons~\eqref{equ:relation_dias_gamma_3_concise},
    check out in $\MProjVersPluri(\Mca)$.
    \smallskip

    Finally, again by~\eqref{equ:relation_algebre_multiassoc}
    and~\eqref{equ:relation_algebre_multiproj}, we have
    \begin{equation} \begin{split}
        x \GDias_a (y \GDias_{a'} z)
            & = x \MAs_a \pi_a(y \MAs_{a'} \pi_{a'}(z)) \\
            & = x \MAs_a \pi_a(y) \MAs_{a'} \pi_{a \Max a'}(z) \\
            & = x \MAs_a \pi_a(y) \MAs_{a \Max a'} \pi_{a \Max a'}(z) \\
            & = (x \GDias_a y) \GDias_{a \Max a'} z,
    \end{split} \end{equation}
    showing that~\eqref{equ:relation_dias_gamma_4_concise},
    and for the same
    reasons~\eqref{equ:relation_dias_gamma_5_concise}, are
    satisfied in $\MProjVersPluri(\Mca)$.
\end{proof}
\medskip

When $\Mca$ is commutative, since for all $x, y \in \MProjVersPluri(\Mca)$
and $a \in [\gamma]$,
\begin{equation}
    x \GDias_a y = x \MAs_a \pi_a(y) =
    \pi_a(y) \MAs_a x = y \DDias_a x,
\end{equation}
it appears that $\MProjVersPluri(\Mca)$ is a commutative
$\gamma$-pluriassociative algebra.
\medskip

When $\Mca$ is unital, $\MProjVersPluri(\Mca)$ has several properties,
summarized in the next proposition.
\medskip

\begin{Proposition}%
\label{prop:algebre_multiprojections_vers_dias_gamma_proprietes}
    Let $\gamma \geq 0$ be an integer, $\Mca$ be a unital
    $\gamma$-multiprojection algebra of height~$h$.
    Then, by denoting by $\Unite$ the unit of $\Mca$ and by $\pi_a$,
    $a \in [\gamma]$, its endomorphisms,
    \begin{enumerate}[label={\it (\roman*)}]
        \item
        \label{item:algebre_multiprojections_vers_dias_gamma_proprietes_1}
        for any $a \in [h]$, $\Unite$ is an $a$-bar-unit of
        $\MProjVersPluri(\Mca)$;
        \item
        \label{item:algebre_multiprojections_vers_dias_gamma_proprietes_2}
        for any $a \leq b \in [h]$,
        $\Halo_a(\MProjVersPluri(\Mca))$ is a subset of
        $\Halo_b(\MProjVersPluri(\Mca))$;
        \item
        \label{item:algebre_multiprojections_vers_dias_gamma_proprietes_3}
        for any $a \in [h]$, the linear span of $\Halo_a(\MProjVersPluri(\Mca))$
        forms an $h\!-\!a\!+\!1$-pluriassociative subalgebra of the
        $h\!-\!a\!+\!1$-pluriassociative subalgebra of
        $\MProjVersPluri(\Mca)$ induced by $[a, h]$;
        \item
        \label{item:algebre_multiprojections_vers_dias_gamma_proprietes_4}
        for any $a \in [h]$, $\pi_a$
        is the identity map if and only if $\Unite$ is an $a$-wire-unit
        of $\MProjVersPluri(\Mca)$.
    \end{enumerate}
\end{Proposition}
\begin{proof}
    Let us denote by $\MAs_a$, $a \in [\gamma]$, the operations of
    $\Mca$.
    \smallskip

    Since $\Unite$ is an $h$-unit of $\Mca$, for all elements $x$ of
    $\MProjVersPluri(\Mca)$ and all $a \in [h]$,
    \begin{equation}
        x \GDias_a \Unite
            = x \MAs_a \pi_a(\Unite) = x \MAs_a \Unite
            = x =
            \Unite \MAs_a x = \pi_a(\Unite) \MAs_a x
        = \Unite \DDias_a x,
    \end{equation}
    showing~\ref{item:algebre_multiprojections_vers_dias_gamma_proprietes_1}.
    \smallskip

    Assume that $e$ is an element of $\Halo_a(\MProjVersPluri(\Mca))$
    for an $a \in [h]$, that is, $e$ is an $a$-bar-unit of
    $\MProjVersPluri(\Mca)$. Then, for all elements $x$ of
    $\MProjVersPluri(\Mca)$,
    \begin{equation}
        x \GDias_a e
            = x \MAs_a \pi_a(e) = x = \pi_a(e) \MAs_a x
        = e \DDias_a x,
    \end{equation}
    showing that $\pi_a(e)$ is the unit for the operation $\MAs_a$ on
    $\MProjVersPluri(\Mca)$
    and therefore, $\pi_a(e) = \Unite$. Since $\Mca$ is unital,
    we have $\pi_b(\Unite) = \Unite$ for all $b \in [h]$. Hence,
    and by~\eqref{equ:relation_algebre_multiproj},
    for all $a \leq b \in [h]$,
    \begin{equation}
        \pi_b(e) = \pi_b(\pi_a(e)) = \pi_b(\Unite) = \Unite.
    \end{equation}
    Then, for all elements $x$ of $\MProjVersPluri(\Mca)$ and all
    $a \leq b \in [h]$,
    \begin{equation}
        x \GDias_b e
            = x \MAs_b \pi_b(e) = x \MAs_b \Unite = x
            = \Unite \MAs_b x = \pi_b(e) \MAs_b x
        = e \DDias_b x,
    \end{equation}
    showing that $e$ is also a $b$-bar-unit of $\MProjVersPluri(\Mca)$,
    whence~\ref{item:algebre_multiprojections_vers_dias_gamma_proprietes_2}.
    \smallskip

    Let $a \in [\gamma]$ and $e$ and $e'$ be elements of
    $\Halo_a(\MProjVersPluri(\Mca))$.
    By~\ref{item:algebre_multiprojections_vers_dias_gamma_proprietes_2},
    $e$ and $e'$ are $b$-bar-units of $\MProjVersPluri(\Mca)$ for all
    $a \leq b \in [h]$ and hence,
    \begin{equation}
        e \GDias_b e' = e = e' \DDias_b e.
    \end{equation}
    Therefore, the linear span of $\Halo_a(\MProjVersPluri(\Mca))$ is
    stable for the operations $\GDias_b$ and $\DDias_b$. This
    implies~\ref{item:algebre_multiprojections_vers_dias_gamma_proprietes_3}.
    \smallskip

    Finally, assume that $\pi_a$ is the identity map for an $a \in [h]$.
    Then, for all elements $x$ of $\MProjVersPluri(\Mca)$,
    \begin{equation}
        \Unite \GDias_a x
            = \Unite \MAs_a \pi_a(x) = \Unite \MAs_a x
            = x
            = x \MAs_a \Unite = \pi_a(x) \MAs_a \Unite
        = x \DDias_a \Unite,
    \end{equation}
    showing that $\Unite$ is an $a$-wire unit of $\MProjVersPluri(\Mca)$.
    Conversely, if $\Unite$ is an $a$-wire unit of $\MProjVersPluri(\Mca)$,
    for all elements $x$ of $\MProjVersPluri(\Mca)$, the relations
    $\Unite \GDias_a x = x = x \DDias_a \Unite$ imply
    $\Unite \MAs_a \pi_a(x) = x = \pi_a(x) \MAs_a \Unite$ and hence,
    $\pi_a(x) = x$. This
    shows~\ref{item:algebre_multiprojections_vers_dias_gamma_proprietes_4}.
\end{proof}
\medskip

\subsubsection{Examples of constructions of pluriassociative algebras}
The construction $\MProjVersPluri$ of
Theorem~\ref{thm:algebre_multiprojections_vers_dias_gamma} allows
to build several $\gamma$-pluriassociative algebras. Here follows few
examples.
\medskip

\paragraph{\bf The $\gamma$-pluriassociative algebra of positive integers}
Let $\gamma \geq 1$ be an integer and consider the vector space $\AlgPos$
of positive integers, endowed with the operations $\MAs_a$, $a \in [\gamma]$,
all equal to the operation $\Max$ extended by linearity and with the
endomorphisms $\pi_a$, $a \in [\gamma]$, linearly defined for any
positive integer $x$ by $\pi_a(x) := a \Max x$. Then, $\AlgPos$ is a
non-unital $\gamma$-multiprojection algebra. By
Theorem~\ref{thm:algebre_multiprojections_vers_dias_gamma},
$\MProjVersPluri(\AlgPos)$ is a $\gamma$-pluriassociative algebra. We
have for instance
\begin{equation}
    \textcolor{Bleu}{2} \GDias_3 \textcolor{Rouge}{5} =
    \textcolor{Rouge}{5},
\end{equation}
and
\begin{equation}
    \textcolor{Bleu}{1} \DDias_3 \textcolor{Rouge}{2} = 3.
\end{equation}
We can observe that $\MProjVersPluri(\AlgPos)$ is commutative, pure, and
its $1$-halo is $\{1\}$. Moreover, when
$\gamma \geq 2$, $\MProjVersPluri(\AlgPos)$ has no wire-unit and no
$a$-bar-unit for $a \geq 2 \in [\gamma]$. This example is important
because it provides a counterexample
for~\ref{item:algebre_multiprojections_vers_dias_gamma_proprietes_2} of
Proposition~\ref{prop:algebre_multiprojections_vers_dias_gamma_proprietes}
in the case when the construction $\MProjVersPluri$ is applied
to a non-unital $\gamma$-multiprojection algebra.
\medskip

\paragraph{\bf The $\gamma$-pluriassociative algebra of finite sets}
Let $\gamma \geq 1$ be an integer and consider the vector space $\AlgEns$
of finite sets of positive integers, endowed with the operations $\MAs_a$,
$a \in [\gamma]$, all equal to the union operation $\cup$ extended by
linearity and with the endomorphisms $\pi_a$, $a \in [\gamma]$, linearly
defined for any finite set of positive integers $x$ by
$\pi_a(x) := x \cap [a, \gamma]$. Then, $\AlgEns$ is a $\gamma$-multiprojection
algebra. By Theorem~\ref{thm:algebre_multiprojections_vers_dias_gamma},
$\MProjVersPluri(\AlgEns)$ is a $\gamma$-pluriassociative algebra. We
have for instance
\begin{equation}
    \{\textcolor{Bleu}{2}, \textcolor{Bleu}{4}\}
    \GDias_3
    \{\textcolor{Rouge}{1}, \textcolor{Rouge}{3}, \textcolor{Rouge}{5}\}
    = \{\textcolor{Bleu}{2}, \textcolor{Rouge}{3},
    \textcolor{Bleu}{4}, \textcolor{Rouge}{5}\},
\end{equation}
and
\begin{equation}
    \{\textcolor{Bleu}{1}, \textcolor{Bleu}{2},
    \textcolor{Bleu}{4}\}
    \DDias_3
    \{\textcolor{Rouge}{1}, \textcolor{Rouge}{3}, \textcolor{Rouge}{5}\}
    = \{\textcolor{Rouge}{1}, \textcolor{Rouge}{3},
    \textcolor{Bleu}{4}, \textcolor{Rouge}{5}\}.
\end{equation}
We can observe that $\MProjVersPluri(\AlgEns)$ is commutative and pure.
Moreover, $\emptyset$ is a $1$-wire-unit of $\MProjVersPluri(\AlgEns)$ and,
by Proposition~\ref{prop:unite_fil}, it is its only wire-unit. Therefore,
$\MProjVersPluri(\AlgEns)$ has height $1$. Observe that for any
$a \in [\gamma]$, the $a$-halo of $\MProjVersPluri(\AlgEns)$ consists
in the subsets of $[a - 1]$. Besides, since $\AlgEns$ is a unital
$\gamma$-multiprojection algebra, $\MProjVersPluri(\AlgEns)$ satisfies
all properties exhibited by
Proposition~\ref{prop:algebre_multiprojections_vers_dias_gamma_proprietes}.
\medskip

\paragraph{\bf The $\gamma$-pluriassociative algebra of words}
Let $\gamma \geq 1$ be an integer and consider the vector space $\AlgMots$
of the words of positive integers. Let us endow $\AlgMots$ with the
operations $\MAs_a$, $a \in [\gamma]$, all equal to the concatenation
operation extended by linearity and with the endomorphisms $\pi_a$,
$a \in [\gamma]$, where for any word $x$ of positive integers, $\pi_a(x)$
is the longest subword of $x$ consisting in letters greater than or equal
to $a$. Then, $\AlgMots$ is a $\gamma$-multiprojection algebra. By
Theorem~\ref{thm:algebre_multiprojections_vers_dias_gamma},
$\MProjVersPluri(\AlgMots)$ is a $\gamma$-pluriassociative algebra. We
have for instance
\begin{equation}
    \textcolor{Bleu}{412} \GDias_3
    \textcolor{Rouge}{14231} =
    \textcolor{Bleu}{412}\textcolor{Rouge}{43},
\end{equation}
and
\begin{equation}
    \textcolor{Bleu}{11} \DDias_2 \textcolor{Rouge}{323} =
    \textcolor{Rouge}{323}.
\end{equation}
We can observe that $\MProjVersPluri(\AlgMots)$ is not commutative and is pure.
Moreover, $\epsilon$ is a $1$-wire-unit of $\MProjVersPluri(\AlgMots)$ and
by Proposition~\ref{prop:unite_fil}, it is its only wire-unit. Therefore,
$\MProjVersPluri(\AlgMots)$ has height $1$. Observe that for any
$a \in [\gamma]$, the $a$-halo of $\MProjVersPluri(\AlgMots)$ consists in
the words on the alphabet $[a - 1]$. Besides, since $\AlgMots$ is a unital
$\gamma$-multiprojection algebra, $\MProjVersPluri(\AlgMots)$ satisfies
all properties exhibited by
Proposition~\ref{prop:algebre_multiprojections_vers_dias_gamma_proprietes}.
\medskip

The $\gamma$-pluriassociative algebras $\MProjVersPluri(\AlgEns)$ and
$\MProjVersPluri(\AlgMots)$ are related in the following way. Let
$I_{\rm com}$ be the subspace of $\MProjVersPluri(\AlgMots)$ generated
by the $x - x'$ where $x$ and $x'$ are words of positive integers and
have the same commutative image. Since $I_{\rm com}$ is a
$\gamma$-pluriassociative algebra ideal of $\MProjVersPluri(\AlgMots)$,
one can consider the quotient $\gamma$-pluriassociative algebra
$\AlgMotsCom := \MProjVersPluri(\AlgMots)/_{I_{\rm com}}$. Its elements
can be seen as commutative words of positive integers.
\medskip

Moreover, let $I_{\rm occ}$ be the subspace of
$\MProjVersPluri(\AlgMotsCom)$ generated by the $x - x'$ where $x$ and $x'$
are commutative words of positive integers and for any letter
$a \in [\gamma]$, $a$ appears in $x$ if and only if $a$ appears in $x'$.
Since $I_{\rm occ}$ is a $\gamma$-pluriassociative algebra ideal of
$\MProjVersPluri(\AlgMotsCom)$, one can consider the quotient
$\gamma$-pluriassociative algebra
$\MProjVersPluri(\AlgMotsCom)/_{I_{\rm occ}}$. Its elements can be seen
as finite subsets of positive integers and we observe that
$\MProjVersPluri(\AlgMotsCom)/_{I_{\rm occ}} = \MProjVersPluri(\AlgEns)$.
\medskip

\paragraph{\bf The $\gamma$-pluriassociative algebra of marked words}
Let $\gamma \geq 1$ be an integer and consider the vector space
$\AlgMotsMarq$ of the words of positive integers where letters can be
marked or not, with at least one occurrence of a marked letter. We
denote by $\bar a$ any {\em marked letter} $a$ and we say that the
{\em value} of $\bar a$ is $a$. Let us endow $\AlgMotsMarq$ with the
linear operations $\MAs_a$, $a \in [\gamma]$, where for all words $u$
and $v$ of $\AlgMotsMarq$, $u \MAs_a v$ is obtained by concatenating $u$
and $v$, and by replacing therein all marked letters by $\bar c$ where
$c := \max(u) \Max a \Max \max(v)$ where $\max(u)$ (resp. $\max(v)$)
denotes the greatest value among the marked letters of $u$ (resp. $v$).
For instance,
\begin{equation}
    \textcolor{Bleu}{2 \bar 1 3 1 \bar 3}
    \MAs_2
    \textcolor{Rouge}{3 \bar 4 \bar 1 2 1}
    =
    \textcolor{Bleu}{2 {\bf \bar 4} 3 1 {\bf \bar 4}}
    \textcolor{Rouge}{3 \bar 4 {\bf \bar 4} 2 1},
\end{equation}
and
\begin{equation}
    \textcolor{Bleu}{\bar 2 1 1 \bar 1}
    \MAs_3
    \textcolor{Rouge}{3 4 \bar 2} =
    \textcolor{Bleu}{{\bf \bar 3} 1 1 {\bf \bar 3}}
    \textcolor{Rouge}{3 4 {\bf \bar 3}}.
\end{equation}
We also endow $\AlgMotsMarq$ with the endomorphisms $\pi_a$,
$a \in [\gamma]$, where for any word $u$ of $\AlgMotsMarq$, $\pi_a(u)$
is obtained by replacing in $u$ any occurrence of a nonmarked letter
smaller than $a$ by $a$. For instance,
\begin{equation}
    \pi_3\left(\textcolor{Rouge}{2} \textcolor{Bleu}{\bar 2}
        \textcolor{Rouge}{14} \textcolor{Bleu}{\bar 4}
        \textcolor{Rouge}{3} \textcolor{Bleu}{\bar 5}\right)
    = \textcolor{Rouge}{\bf 3} \textcolor{Bleu}{\bar 2}
        \textcolor{Rouge}{{\bf 3} 4} \textcolor{Bleu}{\bar 4}
        \textcolor{Rouge}{\bf 3} \textcolor{Bleu}{\bar 5}.
\end{equation}
One can show without difficulty that $\AlgMotsMarq$ is a
$\gamma$-multiprojection algebra. By
Theorem~\ref{thm:algebre_multiprojections_vers_dias_gamma},
$\MProjVersPluri(\AlgMotsMarq)$ is a $\gamma$-pluriassociative algebra.
We have for instance
\begin{equation}
    \textcolor{Bleu}{3 \bar 2 5}
    \GDias_3
    \textcolor{Rouge}{4 \bar 4 1}
    = \textcolor{Bleu}{3 {\bf \bar 4} 5}
        \textcolor{Rouge}{4 \bar 4 {\bf 3}},
\end{equation}
and
\begin{equation}
    \textcolor{Bleu}{1 \bar 3 4 \bar 1 3}
    \DDias_2
    \textcolor{Rouge}{3 1 \bar 2 3 \bar 1 1}
    = \textcolor{Bleu}{{\bf 2} \bar 3 4 {\bf \bar 3} 3}
      \textcolor{Rouge}{3 1 {\bf \bar 3} 3 {\bf \bar 3} 1}.
\end{equation}
We can observe that $\MProjVersPluri(\AlgMotsMarq)$ is not
commutative, pure, and has no wire-units neither bar-units.
\medskip

\paragraph{\bf The free $\gamma$-pluriassociative algebra over one generator}
Let $\gamma \geq 0$ be an integer. We give here a construction of the
free $\gamma$-pluriassociative algebra $\AlgLibre_{\Dias_\gamma}$ over one
generator described in Section~\ref{subsubsec:algebre_dias_gamma_libre}
passing through the following $\gamma$-multiprojection algebra
and the construction $\MProjVersPluri$. Consider the vector space of
nonempty words on the alphabet $\{0\} \cup [\gamma]$ with exactly one
occurrence of $0$, endowed with the operations $\MAs_a$, $a \in [\gamma]$,
all equal to the concatenation operation extended by linearity and with
the endomorphisms $\Augm_a$, $a \in [\gamma]$, defined in
Section~\ref{subsubsec:algebre_dias_gamma_libre}. This vector
space is a $\gamma$-multiprojection algebra. Therefore, by
Theorem~\ref{thm:algebre_multiprojections_vers_dias_gamma}, it gives
rise by the construction $\MProjVersPluri$ to a $\gamma$-pluriassociative
algebra and it appears that it is $\AlgLibre_{\Dias_\gamma}$. Besides,
we can now observe that $\AlgLibre_{\Dias_\gamma}$ is not commutative,
pure, and has no wire-units neither bar-units.
\medskip

\section{Pluritriassociative operads}%
\label{sec:pluritriass}
Our original idea of using the $\T$ construction (see
Sections~\ref{subsec:monoides_vers_operades}
and~\ref{subsubsec:construction_dias_gamma}) to obtain a
generalization of the diassociative operad admits an analogue in the
context of the triassociative operad~\cite{LR04}. We describe in this
section a generalisation on a nonnegative integer parameter $\gamma$
of the triassociative operad.
\medskip

Since the proofs of the results contained in this section are very
similar to the ones of Section~\ref{sec:dias_gamma}, we omit proofs here.
\medskip

\subsection{Construction and first properties}%
\label{subsec:trias_gamma}
For any integer $\gamma \geq 0$, we define $\Trias_\gamma$ as the
suboperad of $\Mca_\gamma$ generated by
\begin{equation} \label{equ:generateurs_trias_gamma}
   \{0a, 00, a0 : a \in [\gamma]\}.
\end{equation}
By definition, $\Trias_\gamma$ is the vector space of words that can be
obtained by partial compositions of words
of~\eqref{equ:generateurs_trias_gamma}. We have, for instance,
\begin{equation}
    \Trias_2(1) = \Vect(\{0\}),
\end{equation}
\begin{equation}
    \Trias_2(2) = \Vect(\{00, 01, 02, 10, 20\}),
\end{equation}
\begin{multline}
    \Trias_2(3)
    = \Vect(\{000, 001, 002, 010, 011, 012, 020, 021, \\
        022, 100, 101, 102, 110, 120, 200, 201, 202, 210, 220\}),
\end{multline}
\medskip

It follows immediately from the definition of $\Trias_\gamma$ as a
suboperad of $\T \Mca_\gamma$ that $\Trias_\gamma$ is a set-operad.
Moreover, one can observe that  $\Trias_\gamma$ is generated by the same
generators as the ones of $\Dias_\gamma$
(see~\eqref{equ:generateurs_dias_gamma}), plus the word $00$. Therefore,
$\Dias_\gamma$ is a suboperad of $\Trias_\gamma$. Besides, note that
$\Trias_0$ is the associative operad and that $\Trias_\gamma$ is a
suboperad of $\Trias_{\gamma + 1}$. We call $\Trias_\gamma$ the
{\em $\gamma$-pluritriassociative operad}.
\medskip

\begin{Proposition} \label{prop:elements_trias_gamma}
    For any integer $\gamma \geq 0$, as a set-operad, the underlying set
    of $\Trias_\gamma$ is the set of the words on the alphabet
    $\{0\} \cup [\gamma]$ containing at least one occurrence of $0$.
\end{Proposition}
\medskip

We deduce from Proposition~\ref{prop:elements_trias_gamma} that the
Hilbert series of $\Trias_\gamma$ satisfies
\begin{equation}
    \Hca_{\Trias_\gamma}(t) =
    \frac{t}{(1 - \gamma t)(1 - \gamma t - t)}
\end{equation}
and that for all $n \geq 1$,
$\dim \Trias_\gamma(n) = (\gamma + 1)^n - \gamma^n$. For instance, the
first dimensions of  $\Trias_1$, $\Trias_2$, $\Trias_3$, and $\Trias_4$
are respectively
\begin{equation}
    1, 3, 7, 15, 31, 63, 127, 255, 511, 1023, 2047,
\end{equation}
\begin{equation}
    1, 5, 19, 65, 211, 665, 2059, 6305, 19171, 58025, 175099,
\end{equation}
\begin{equation}
    1, 7, 37, 175, 781, 3367, 14197, 58975, 242461, 989527, 4017157,
\end{equation}
\begin{equation}
    1, 9, 61, 369, 2101, 11529, 61741, 325089, 1690981, 8717049, 44633821.
\end{equation}
The first one is Sequence~\Sloane{A000225}, the second one is
Sequence~\Sloane{A001047}, the third one is Sequence~\Sloane{A005061},
and the last one is Sequence~\Sloane{A005060} of~\cite{Slo}.
\medskip

\subsection{Presentation by generators and relations}
We follow the same strategy as the one used in
Section~\ref{subsec:presentation_dias_gamma} to establish a presentation
by generators and relations of $\Trias_\gamma$ and prove that it is a
Koszul operad. As announced above, we omit complete proofs here but
we describe the analogue for $\Trias_\gamma$ of the maps $\Mot_\gamma$
and $\Equerre_\gamma$ defined in
Section~\ref{subsec:presentation_dias_gamma} for the operad $\Dias_\gamma$.
\medskip

For any integer $\gamma \geq 0$, let $\GenTrias := \GenTrias(2)$ be the
graded set where
\begin{equation}
    \GenTrias(2) := \{\GDias_a, \MTrias, \DDias_a : a \in [\gamma] \}.
\end{equation}
\medskip

Let $\Tfr$ be a syntax tree of $\OpLibre\left(\GenTrias\right)$ and $x$ be
a leaf of $\Tfr$. We say that an integer $a \in \{0\} \cup [\gamma]$ is
{\em eligible} for $x$ if $a = 0$ or there is an ancestor $y$ of $x$
labeled by $\GDias_a$ (resp. $\DDias_a$) and $x$ is in the right (resp.
left) subtree of $y$. The {\em image} of $x$ is its greatest eligible
integer. Moreover, let
\begin{equation}
    \MotT_\gamma : \OpLibre\left(\GenTrias\right)(n) \to \Trias_\gamma(n),
    \qquad n \geq 1,
\end{equation}
the map where $\MotT_\gamma(\Tfr)$ is the word obtained by considering,
from left to right, the images of the leaves of $\Tfr$
(see Figure~\ref{fig:exemple_mot_gamma_trias_gamma}).
\begin{figure}[ht]
    \centering
     \begin{tikzpicture}[xscale=.28,yscale=.18]
        \node(0)at(0.00,-15.33){};
        \node(10)at(10.00,-11.50){};
        \node(12)at(12.00,-15.33){};
        \node(14)at(14.00,-15.33){};
        \node(16)at(16.00,-19.17){};
        \node(18)at(18.00,-19.17){};
        \node(2)at(2.00,-15.33){};
        \node(20)at(20.00,-15.33){};
        \node(22)at(22.00,-11.50){};
        \node(4)at(4.00,-11.50){};
        \node(6)at(6.00,-11.50){};
        \node(8)at(8.00,-11.50){};
        \node(1)at(1.00,-11.50){\begin{math}\GDias_1\end{math}};
        \node(3)at(3.00,-7.67){\begin{math}\DDias_3\end{math}};
        \node(5)at(5.00,-3.83){\begin{math}\GDias_4\end{math}};
        \node(7)at(7.00,-7.67){\begin{math}\MTrias\end{math}};
        \node(9)at(9.00,0.00){\begin{math}\DDias_2\end{math}};
        \node(11)at(11.00,-7.67){\begin{math}\GDias_3\end{math}};
        \node(13)at(13.00,-11.50){\begin{math}\DDias_4\end{math}};
        \node(15)at(15.00,-3.83){\begin{math}\MTrias\end{math}};
        \node(17)at(17.00,-15.33){\begin{math}\DDias_3\end{math}};
        \node(19)at(19.00,-11.50){\begin{math}\DDias_2\end{math}};
        \node(21)at(21.00,-7.67){\begin{math}\GDias_1\end{math}};
        \draw(0)--(1);\draw(1)--(3);\draw(10)--(11);\draw(11)--(15);
        \draw(12)--(13);\draw(13)--(11);\draw(14)--(13);\draw(15)--(9);
        \draw(16)--(17);\draw(17)--(19);\draw(18)--(17);\draw(19)--(21);
        \draw(2)--(1);\draw(20)--(19);\draw(21)--(15);\draw(22)--(21);
        \draw(3)--(5);\draw(4)--(3);\draw(5)--(9);\draw(6)--(7);
        \draw(7)--(5);\draw(8)--(7);
        \node(r)at(9.00,3){};
        \draw(r)--(9);
        \node[below of=0,node distance=3mm]
            {\small \begin{math}\textcolor{Bleu}{3}\end{math}};
        \node[below of=2,node distance=3mm]
            {\small \begin{math}\textcolor{Bleu}{3}\end{math}};
        \node[below of=4,node distance=3mm]
            {\small \begin{math}\textcolor{Bleu}{2}\end{math}};
        \node[below of=6,node distance=3mm]
            {\small \begin{math}\textcolor{Bleu}{4}\end{math}};
        \node[below of=8,node distance=3mm]
            {\small \begin{math}\textcolor{Bleu}{4}\end{math}};
        \node[below of=10,node distance=3mm]
            {\small \begin{math}\textcolor{Bleu}{0}\end{math}};
        \node[below of=12,node distance=3mm]
            {\small \begin{math}\textcolor{Bleu}{4}\end{math}};
        \node[below of=14,node distance=3mm]
            {\small \begin{math}\textcolor{Bleu}{3}\end{math}};
        \node[below of=16,node distance=3mm]
            {\small \begin{math}\textcolor{Bleu}{3}\end{math}};
        \node[below of=18,node distance=3mm]
            {\small \begin{math}\textcolor{Bleu}{2}\end{math}};
        \node[below of=20,node distance=3mm]
            {\small \begin{math}\textcolor{Bleu}{0}\end{math}};
        \node[below of=22,node distance=3mm]
            {\small \begin{math}\textcolor{Bleu}{1}\end{math}};
    \end{tikzpicture}
    \caption{A syntax tree $\Tfr$ of $\OpLibre\left(\GenTrias\right)$
    where images of its leaves are shown. This tree satisfies
    $\MotT_\gamma(\Tfr) = \textcolor{Bleu}{332440433201}$.}
    \label{fig:exemple_mot_gamma_trias_gamma}
\end{figure}
Observe that $\MotT_\gamma$ is an extension of $\Mot_\gamma$
(see~\eqref{equ:application_mot_gamma}).
\medskip

Consider now the map
\begin{equation}
    \EquerreT_\gamma :
    \Trias_\gamma(n) \to \OpLibre\left(\GenTrias\right)(n),
    \qquad n \geq 1,
\end{equation}
defined for any word $x$ of $\Trias_\gamma$ by
\begin{equation} \label{equ:definition_equerre_trias_gamma}
    \begin{split}\EquerreT_\gamma(x)\end{split} :=
    \begin{split}
    \begin{tikzpicture}[xscale=.5,yscale=.45]
        \node(0)at(0.00,-5.40){};
        \node(2)at(3.00,-7.20){};
        \node(4)at(5.00,-7.20){};
        \node(6)at(6.00,-3.60){};
        \node(8)at(9.00,-1.80){};
        \node(1)at(3.00,-4){\begin{math}\Equerre_\gamma(u)\end{math}};
        \node(5)at(5.00,-2){\begin{math}\MTrias\end{math}};
        \node(7)at(12.00,2.00){\begin{math}\MTrias\end{math}};
        \node(9)at(7.5,-4){\begin{math}\GDias_{v^{(1)}_{k^{(1)}}}\end{math}};
        \node(10)at(5,-6.5){\begin{math}\GDias_{v^{(1)}_1}\end{math}};
        \node(11)at(14.5,0){\begin{math}\GDias_{v^{(\ell)}_{k^{(\ell)}}}\end{math}};
        \node(12)at(12,-2.5){\begin{math}\GDias_{v^{(\ell)}_1}\end{math}};
        \node(r)at(12,3.5){};
        \draw(1)--(5);
        \draw[densely dashed](5)--(7);
        \draw(9)--(5);
        \draw(11)--(7);
        \draw(7)--(r);
        \draw[densely dashed](9)--(10);
        \draw[densely dashed](11)--(12);
        \node(f1)at(4,-8){}; \draw(f1)--(10);
        \node(f2)at(6,-8){}; \draw(f2)--(10);
        \node(f3)at(8.5,-5.5){}; \draw(f3)--(9);
        \node(f4)at(11,-4){}; \draw(f4)--(12);
        \node(f5)at(13,-4){}; \draw(f5)--(12);
        \node(f6)at(15.5,-1.5){}; \draw(f6)--(11);
    \end{tikzpicture}
    \end{split}\,,
\end{equation}
where $x$ decomposes, by Proposition~\ref{prop:elements_trias_gamma},
uniquely in $x = u0v^{(1)}\dots 0v^{(\ell)}$ where $u$ is a word of
$\Dias_\gamma$ and for all $i \in [\ell]$, the $v^{(i)}$ are words on
the alphabet $[\gamma]$. The length $|v^{(i)}|$ of any $v_i$ is denoted
by $k^{(i)}$. The dashed edges denote left comb trees wherein internal
nodes are labeled as specified. Observe that $\EquerreT_\gamma$ is an
extension of $\Equerre_\gamma$ (see~\eqref{equ:application_equerre_gamma}).
We shall call any syntax tree of the
form~\eqref{equ:definition_equerre_trias_gamma} an
{\em extended hook syntax tree}.
\medskip

\begin{Theoreme} \label{thm:presentation_trias_gamma}
    For any integer $\gamma \geq 0$, the operad $\Trias_\gamma$ admits
    the following presentation. It is generated by $\GenTrias$ and its
    space of relations $\RelTrias$ is the space induced by the
    equivalence relation $\Rel_\gamma$ satisfying
    \begin{subequations}
    \begin{equation}\label{equ:relation_presentation_trias_gamma_1}
        \MTrias \circ_1 \MTrias
        \enspace \Rel_\gamma \enspace
        \MTrias \circ_2 \MTrias,
    \end{equation}
    \begin{equation}\label{equ:relation_presentation_trias_gamma_2}
        \GDias_a \circ_1 \MTrias
        \enspace \Rel_\gamma \enspace
        \MTrias \circ_2 \GDias_a,
        \qquad a \in [\gamma],
    \end{equation}
    \begin{equation}\label{equ:relation_presentation_trias_gamma_3}
        \MTrias \circ_1 \DDias_a
        \enspace \Rel_\gamma \enspace
        \DDias_a \circ_2 \MTrias,
        \qquad a \in [\gamma],
    \end{equation}
    \begin{equation}\label{equ:relation_presentation_trias_gamma_4}
        \MTrias \circ_1 \GDias_a
        \enspace \Rel_\gamma \enspace
        \MTrias \circ_2 \DDias_a,
        \qquad a \in [\gamma],
    \end{equation}
    \begin{equation}\label{equ:relation_presentation_trias_gamma_5}
        \GDias_a \circ_1 \DDias_{a'}
        \enspace \Rel_\gamma \enspace
        \DDias_{a'} \circ_2 \GDias_a,
        \qquad a, a' \in [\gamma],
    \end{equation}
    \begin{equation}\label{equ:relation_presentation_trias_gamma_6}
        \GDias_a \circ_1 \GDias_b
        \enspace \Rel_\gamma \enspace
        \GDias_a \circ_2 \DDias_b,
        \qquad a < b \in [\gamma],
    \end{equation}
    \begin{equation}\label{equ:relation_presentation_trias_gamma_7}
        \DDias_a \circ_1 \GDias_b
        \enspace \Rel_\gamma \enspace
        \DDias_a \circ_2 \DDias_b,
        \qquad a < b \in [\gamma],
    \end{equation}
    \begin{equation}\label{equ:relation_presentation_trias_gamma_8}
        \GDias_b \circ_1 \GDias_a
        \enspace \Rel_\gamma \enspace
        \GDias_a \circ_2 \GDias_b,
        \qquad a < b \in [\gamma],
    \end{equation}
    \begin{equation}\label{equ:relation_presentation_trias_gamma_9}
        \DDias_a \circ_1 \DDias_b
        \enspace \Rel_\gamma \enspace
        \DDias_b \circ_2 \DDias_a,
        \qquad a < b \in [\gamma],
    \end{equation}
    \begin{equation}\label{equ:relation_presentation_trias_gamma_10}
        \GDias_d \circ_1 \GDias_d
        \enspace \Rel_\gamma \enspace
        \GDias_d \circ_2 \MTrias
        \enspace \Rel_\gamma \enspace
        \GDias_d \circ_2 \GDias_c
        \enspace \Rel_\gamma \enspace
        \GDias_d \circ_2 \DDias_c,
        \qquad c \leq d \in [\gamma],
    \end{equation}
    \begin{equation}\label{equ:relation_presentation_trias_gamma_11}
        \DDias_d \circ_1 \GDias_c
        \enspace \Rel_\gamma \enspace
        \DDias_d \circ_1 \DDias_c
        \enspace \Rel_\gamma \enspace
        \DDias_d \circ_1 \MTrias
        \enspace \Rel_\gamma \enspace
        \DDias_d \circ_2 \DDias_d,
        \qquad c \leq d \in [\gamma].
    \end{equation}
    \end{subequations}
\end{Theoreme}
\medskip

Observe that, by Theorem~\ref{thm:presentation_trias_gamma}, $\Trias_1$
and the triassociative operad~\cite{LR04} admit the same presentation.
Then, for all integers $\gamma \geq 0$, the operads $\Trias_\gamma$ are
generalizations of the triassociative operad.
\medskip

\begin{Theoreme} \label{thm:koszulite_trias_gamma}
    For any integer $\gamma \geq 0$, $\Trias_\gamma$ is a Koszul operad.
    Moreover, the set of extended hook syntax trees of
    $\OpLibre\left(\GenTrias\right)$ forms a Poincaré-Birkhoff-Witt
    basis of $\Trias_\gamma$.
\end{Theoreme}
\medskip

\bibliographystyle{alpha}
\bibliography{Bibliographie}

\end{document}